%% file: root.tex
\pdfoutput=1
\documentclass{article}

\PassOptionsToPackage{numbers,sort&compress}{natbib}

 \usepackage[final]{neurips_2020}

\usepackage{algorithm, algpseudocode}
\usepackage[utf8]{inputenc} 
\usepackage[T1]{fontenc}    
\usepackage{hyperref}       
\usepackage{url}            
\usepackage{booktabs}       
\usepackage{amsfonts}       
\usepackage{nicefrac}       
\usepackage{microtype}      
\usepackage{amsthm}  
\usepackage{amsmath}  
\usepackage{amssymb} 
\usepackage{mathtools}
\usepackage{subcaption}
\usepackage{xcolor}
\usepackage{enumitem}
\usepackage{graphicx}
\usepackage{subfiles}
\usepackage{sidecap}
\usepackage{bbm}
\usepackage{xparse}

\usepackage{hyperref}       
\usepackage{xcolor,colortbl}
\definecolor{darkblue}{rgb}{0.0,0.0,0.65}
\definecolor{darkred}{rgb}{0.65,0.0,0.0}
\hypersetup{
  colorlinks = true,
  citecolor  = darkblue,
  linkcolor  = darkred,
  filecolor  = darkblue,
  urlcolor   = darkblue,
}

\include{notations}

\newtheorem{theorem}{Theorem}
\newtheorem{lemma}[theorem]{Lemma}
\newtheorem{corollary}[theorem]{Corollary}
\newtheorem{remark}[theorem]{Remark}

\newtheorem{assumption}{Assumption}

\newcommand{\tsum}{\textstyle\sum} 
\newcommand{\indicator}[1]{\mathbbm{1}_{\{ #1 \}}}
\DeclareMathOperator{\ber}{Bern}
\DeclareMathOperator{\KL}{KL}

\DeclarePairedDelimiterX{\inp}[2]{\langle}{\rangle}{#1, #2}
\DeclarePairedDelimiterX{\abs}[1]{\lvert}{\rvert}{#1}
\DeclarePairedDelimiterX{\norm}[1]{\lVert}{\rVert}{#1}
\DeclarePairedDelimiterX{\cbr}[1]{\{}{\}}{#1} 
\DeclarePairedDelimiterX{\rbr}[1]{(}{)}{#1} 
\DeclarePairedDelimiterX{\sbr}[1]{[}{]}{#1} 
\DeclareCaptionFormat{myformat}{\fontsize{8.5}{10}\selectfont#1#2#3}

\newcommand{\Oc}{\mathcal{O}}
\newcommand{\Ac}{\mathcal{A}}
\newcommand{\Bc}{\mathcal{B}}

\newcommand{\Dc}{\mathcal{D}}
\newcommand{\tDc}{\mathcal{\tilde D}}
\newcommand{\E}{\mathbb{E}}
\newcommand{\reals}{\mathbb{R}}
\newcommand{\inner}[1]{\langle #1 \rangle}

\title{Why are Adaptive Methods Good \\ for Attention Models? }

\author{%
	Jingzhao Zhang \\
	MIT \\
	\texttt{jzhzhang@mit.edu} \\
	\And
	Sai Praneeth Karimireddy \\
	EPFL \\
	\texttt{sai.karimireddy@epfl.ch} \\
	\And
	Andreas Veit \\
	Google Research \\
	\texttt{aveit@google.com} \\
	\And
	Seungyeon Kim \\
	Google Research \\
	\texttt{seungyeonk@google.com} \\
	\And
	Sashank Reddi \\
	Google Research \\
	\texttt{sashank@google.com} \\
	\And
	Sanjiv Kumar \\
	Google Research \\
	\texttt{sanjivk@google.com} \\
	\And
	Suvrit Sra \\
	MIT \\
	suvrit@mit.edu	
}

\begin{document}
	
	\maketitle
	\vspace{-0.1cm}
	\begin{abstract}

	
While stochastic gradient descent (SGD) is still the \emph{de facto} algorithm in deep learning, adaptive methods like Clipped SGD/Adam have been observed to outperform SGD across important tasks, such as attention models. The settings under which SGD performs poorly in comparison to adaptive methods are not well understood yet. In this paper, we provide empirical and theoretical evidence that a heavy-tailed distribution of the noise in stochastic gradients is one cause of SGD's poor performance. We provide the first tight upper and lower convergence bounds for adaptive gradient methods under heavy-tailed noise. Further, we demonstrate how gradient clipping plays a key role in addressing heavy-tailed gradient noise. Subsequently, we show how clipping can be applied in practice by developing an \emph{adaptive} coordinate-wise clipping algorithm (ACClip) and demonstrate its superior performance on BERT pretraining and finetuning tasks.

	\end{abstract}
		\vspace{-0.1cm}

	\input{sections/introduction}
\input{sections/preliminaries}


	\input{sections/main-results}

	\input{sections/experiments}\input{sections/discussion}

{
	\bibliographystyle{abbrvnat}
	\bibliography{bibliography}
}
	\clearpage
	
	\appendix

\input{sections/appendix}

\end{document}

%% file: sections/introduction.tex

\section{Introduction}\label{sec:introduction}
\vspace{-0.2cm}

Stochastic gradient descent (SGD) is the canonical algorithm for training neural networks \citep{robmon51}. SGD iteratively updates model parameters in the negative gradient direction and seamlessly scales to large-scale settings. Though a well-tuned SGD outperforms adaptive methods \citep{wilson2017marginal} in many tasks including ImageNet classification (see Figure~\ref{fig:loss-sgd-Adam-imagenet}), certain tasks necessitate the use of \emph{adaptive} variants of SGD (e.g., Adagrad~\citep{duchi2011adaptive}, Adam~\citep{kingma2014Adam}, AMSGrad~\citep{reddi2018convergence}), which employ adaptive learning rates. For instance, consider training an attention model~\citep{vaswani2017attention} using BERT~\citep{devlin2018bert}. Figure~\ref{fig:loss-sgd-Adam} shows that in spite of extensive hyperparameter tuning, SGD converges much slower than Adam during BERT training.

In this work, we provide one explanation for why adaptivity can facilitate convergence with theoretical and empirical evidence. The significant hint that initializes our work comes from the distribution of the stochastic gradients. For Imagenet, the norms of the mini-batch gradients are typically quite small and well concentrated around their mean. On the other hand, the mini-batch gradient norms for BERT take a wide range of values and are sometimes much larger than their mean value. More formally, while the distribution of the stochastic gradients in Imagenet is well approximated by a Gaussian, the distribution for BERT seems to be \emph{heavy-tailed}. Such observation leads us to the question: does adaptivity stabilize optimization under heavy-tailed noise?

We provide a positive answer to the above question by performing both theoretical and empirical studies of the convergence of optimization methods under heavy-tailed noise. In this setting, some of the stochastic gradients are much larger than the mean and can excessively influence the updates of SGD. This makes SGD unstable and leads to its poor performance. A natural strategy to stabilize the updates is to \emph{clip} the magnitude of the stochastic gradients. We prove that indeed it is sufficient to ensure convergence even under heavy-tailed noise. Based on the analysis, we then motivate the design of a novel algorithm (ACClip) that outperforms ADAM on BERT related tasks. Specifically, we make the following  contributions:

\begin{itemize}
\setlength{\itemsep}{1pt}
    \item We empirically show that in tasks on which Adam outperforms SGD (BERT pretraining), the noise in stochastic gradients is heavy-tailed. On the other hand, on tasks where traditionally SGD outperforms Adam (ImageNet training), we show that the noise is well concentrated. 
    \item In section~\ref{sec:thm}, we study the convergence of gradient methods under heavy-tailed noise condition where SGD's performance degrades and its convergence might fail. We then establish (with upper and lower bounds) the convergence of \emph{clipped} gradient methods under the same condition and prove that they obtain theoretically \emph{optimal} rates.
    \item Though clipping speeds up SGD, it does not close the gap between SGD and ADAM. In section~\ref{sec:transformerxl}, we motivated the  a novel \emph{adaptive}-threshold coordinate-wise clipping algorithm and in section~\ref{sec:experiment} experimentally show that it outperforms Adam on BERT training tasks.
\end{itemize}

\begin{figure}[t]
\centering

\end{figure}

\subsection{Related work}

\textbf{Adaptive step sizes.} Adaptive step sizes during optimization have long been studied \citep{armijo1966minimization,polyak1987introduction}. More recently, \citet{duchi2011adaptive} developed the Adagrad algorithm that benefits from the sparsity in stochastic gradients. Inspired by Adagrad, several adaptive methods have been proposed in the deep learning community \citep{tieleman2012lecture,kingma2014Adam}. Recently, there has been a surge in interest to study the theoretical properties of these adaptive gradient methods due to~\citep{reddi2018convergence}, which pointed out the non-convergence of Adam and proposed an alternative algorithm, AMSGrad. Since then, many works studied different interesting aspects of adaptive methods, see \citep{ward2018adagrad, li2018convergence, zhou2018convergence,staib2019escaping, chen2018convergence, zou2018convergence, zhou2018adashift, agarwal2018case, zou2019sufficient, liu2019variance,ma2019adequacy,huang2018nostalgic, zhang2020why}. Another direction of related work is normalized gradient descent, which has been studied for quasi-convex and non-convex settings \citep{levy2016power, hazan2015beyond}.  In contrast to our work, these prior works assume standard noise distributions that might not be applicable to key modern applications such as attention models, which exhibit heavy-tailed noise. Furthermore, convergence rates of adaptive methods are mostly worse than SGD.

\textbf{Noise in neural network.} There has been little study of the actual stochastic gradient noise distributions in neural network training. To our knowledge, \cite{simsekli2019tail, nguyen2019first, csimcsekli2020fractional} start the topic and observe heavy tailed noise in network training. Our work differs in two important ways: \textit{First}, we treat the noise as a high dimensional vector, while \citep{simsekli2019tail} treat deviations in each coordinate as scaler noises to estimate tail index. Hence, we observe that the example given in \cite{simsekli2019tail} is well-concentrated when viewed as a random vector. This is also confirmed by \cite{panigrahi2019non}. More experimental comparisons are in Appendix~\ref{sec:comp}. \textit{Second}, we focus on convergence of optimization algorithm, the previously mentioned works focus on Langevin dynamics and escaping saddle points. The convergence rate given in \cite{nguyen2019first} is for global Holder-continuous functions, which restricts the function variations and excludes examples like quadratic functions. Our analysis instead provides the first convergence rates under the standard L-smoothness setting. Further, \cite{gorbunov2020stochastic} studies accelerated first order methods under less concentrated noise, however, there ``heavy-tailedness'' refers to non-sub-Gaussianity.

%% file: sections/preliminaries.tex
\vspace{-0.2cm}
\section{Heavy-tailed noise in stochastic gradients}
\vspace{-0.2cm}

To gain intuition about the difference between SGD and adaptive methods, we start our discussion with the study of noise distributions of stochastic gradient that arise during neural network training. In particular, we focus on noise distributions while training two popular deep learning models --- BERT and ResNet. Note that BERT and ResNet are typically trained with Adam and SGD (with momentum) respectively and can thus, provide insights about difference between these optimizers. 

We first investigate the distribution of the gradient noise norm $\norm*{g - \nabla f(x)}$ in the aforementioned neural network models, where $g$ is the stochastic gradient computed from a minibatch sample. In particular, we fix the model at initialization without doing any updates. We then iterate through the dataset to compute the noise norm for each minibatch. Figure~\ref{fig:heavy-noise}~(b) and (f) show these distributions for ResNet50 on ImageNet and BERT on the Wikipedia and books dataset at model initialization respectively.  For comparison, we plot distributions of a normalized sum of squared Gaussians, a well-concentrated distribution, and a Levy-$\alpha$-stable distribution, a heavy-tailed distribution, in Figure~\ref{fig:heavy-noise}~(c) and (g) respectively. We observe that the noise distribution for BERT appears heavy-tailed, while that of ResNet50 is well-concentrated. Results for noise distributions at other stages of training are displayed in Figure~\ref{fig:nonstationary}.

To support this observation, in Figure~\ref{fig:heavy-noise}~(d) and (h) we further show the empirical variance of stochastic gradients with respect to the sample size used in the estimation. The results highlight that while the corresponding estimator converges for Imagenet, the empirical variance does not converge in BERT training even as the sample size approaches $10^7$.

From the obeservation that the noise can be heavy-tailed, we hypothesize that this is one major aspect that determines the performance of SGD and adaptive methods. In the rest of the paper, we argue and provide evidence that adaptive methods can be faster than SGD in scenarios where heavy-tailed noise distributions arise. More experiment details can be found in Section~\ref{sec:experiment}.


\begin{figure}[t]
\centering
\captionsetup[subfigure]{position=b,format=myformat}

	\begin{subfigure}{0.24\textwidth}
	\centering
	\includegraphics[width=0.85\linewidth]{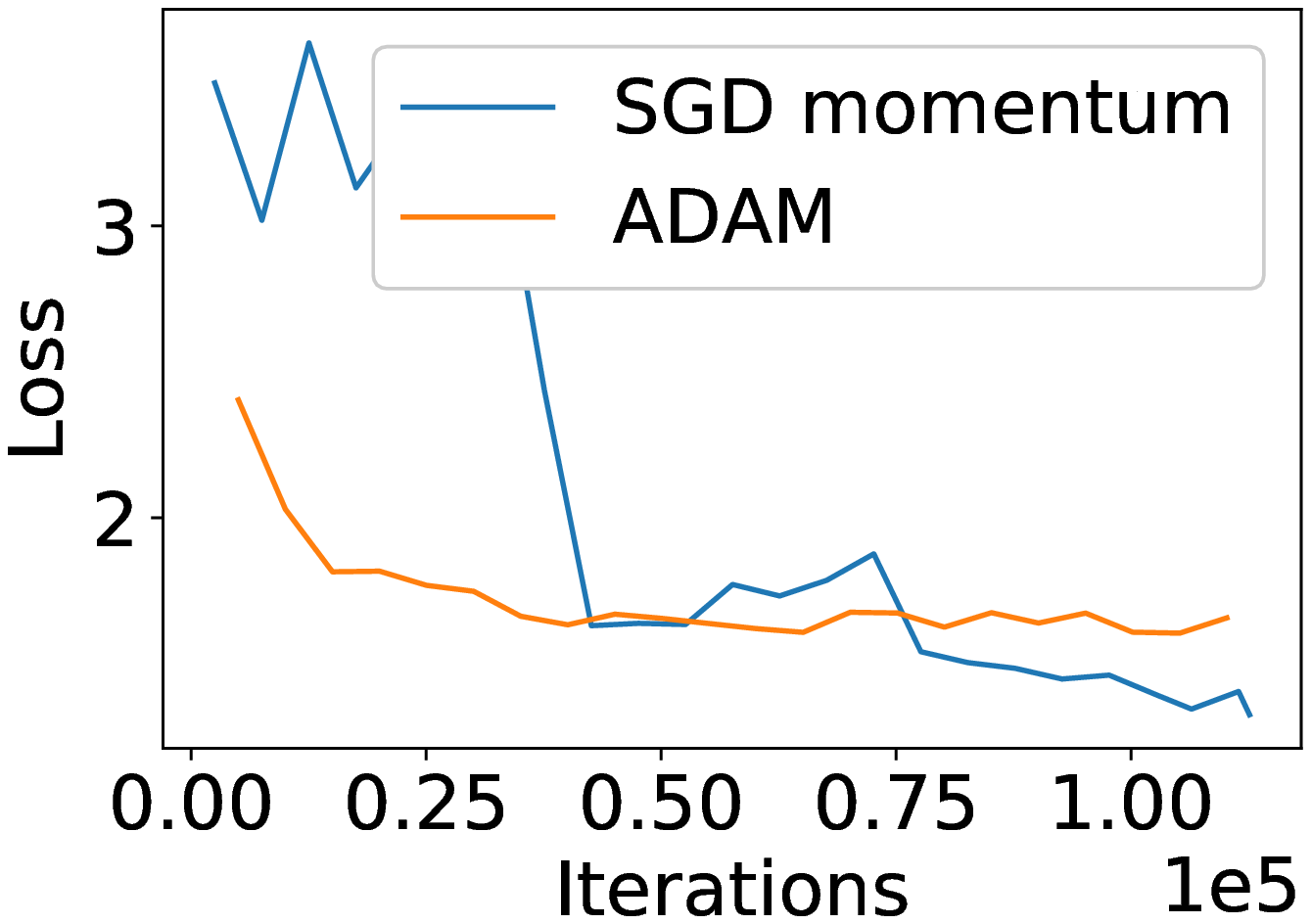}
	\caption{}\label{fig:loss-sgd-Adam-imagenet}
\end{subfigure}
    \begin{subfigure}{.24\textwidth}
	    \centering
    	\includegraphics[width=0.9\linewidth]{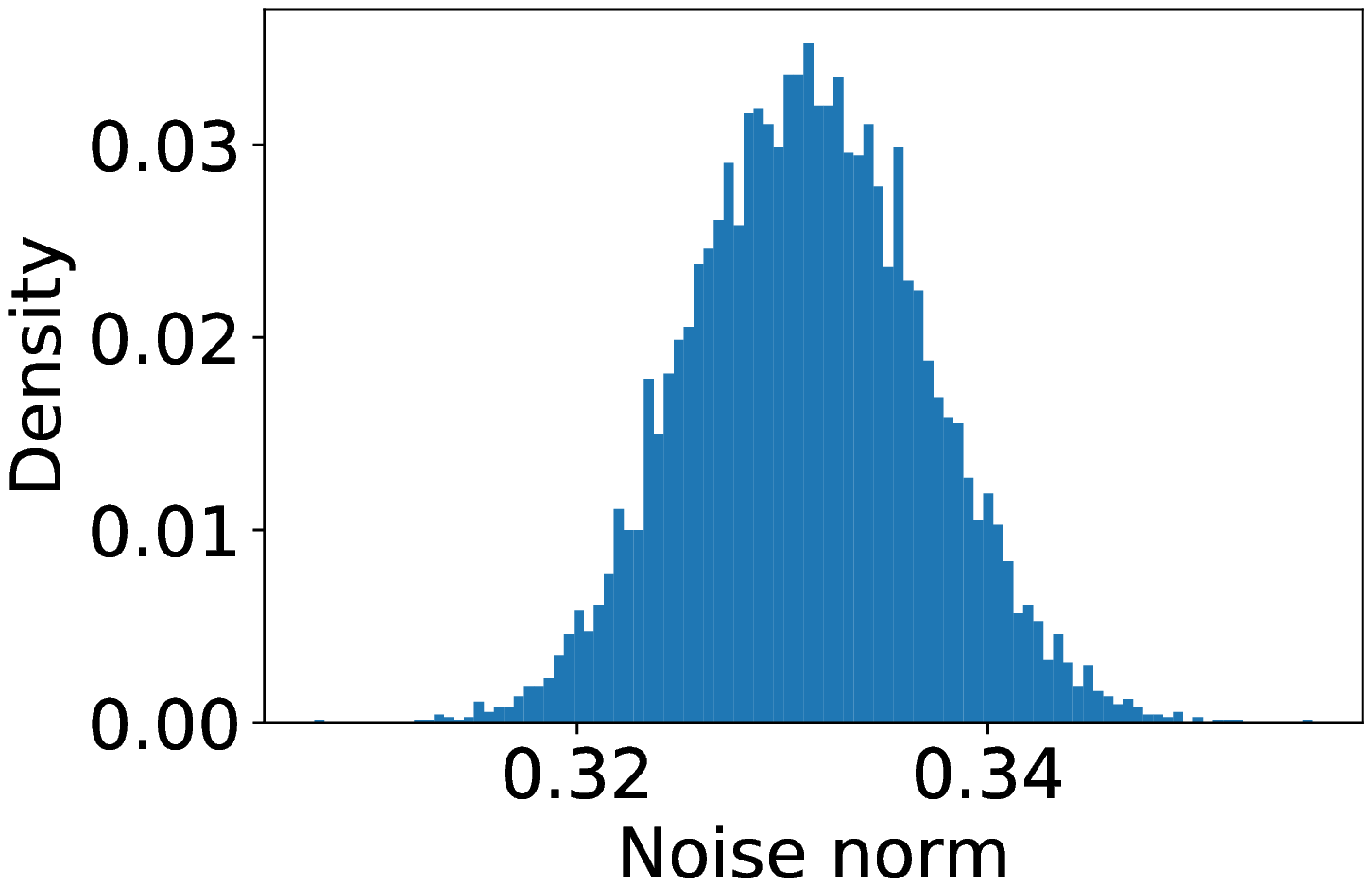}
        \caption{$\text{ImageNet training}$}
    \end{subfigure}
    \begin{subfigure}{.24\textwidth}
	    \centering
    	\includegraphics[width=0.9\linewidth]{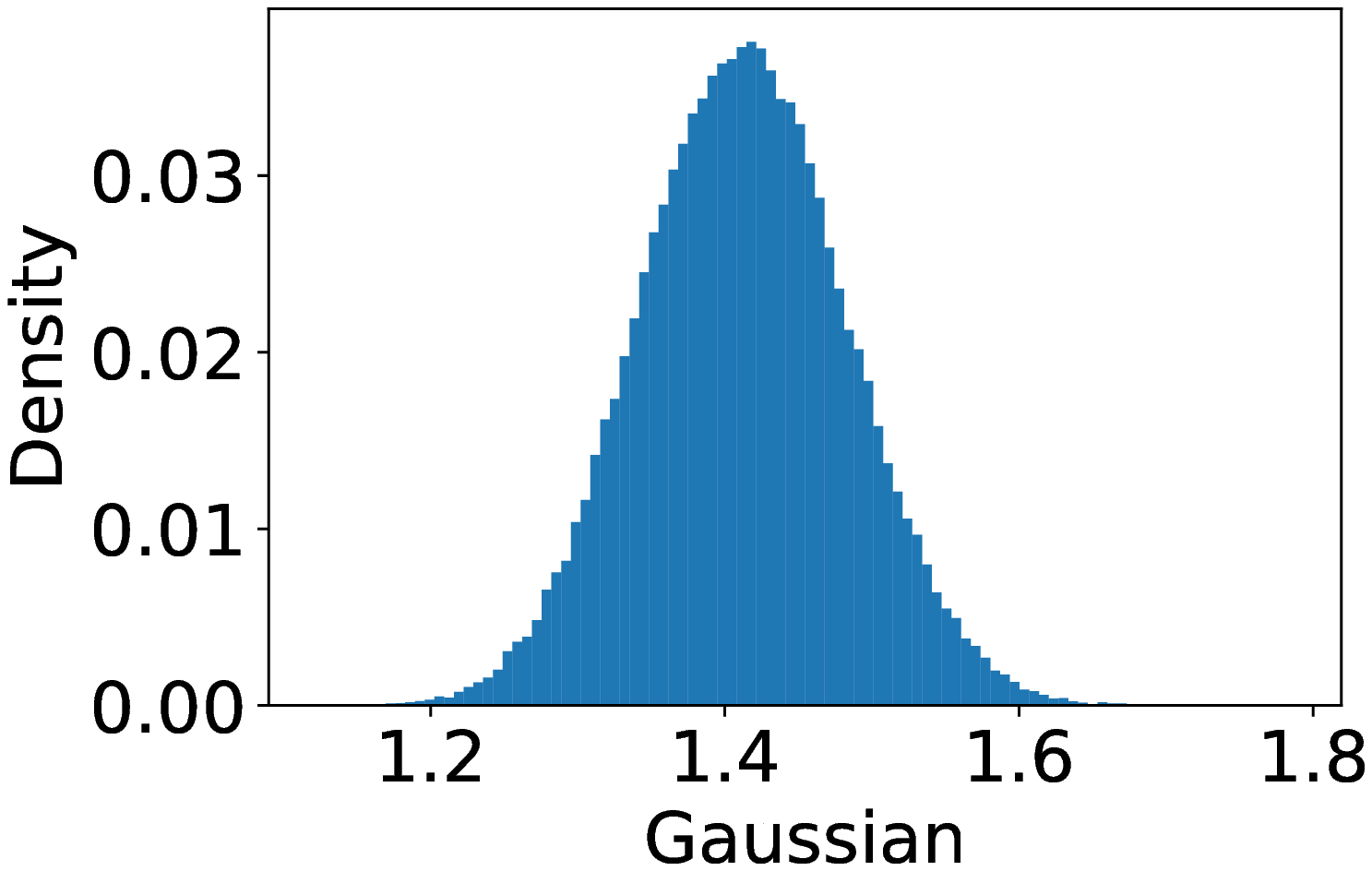}
        \caption{Synthetic Gaussian}
	\end{subfigure}
    \begin{subfigure}{.24\textwidth}
	    \centering
    	\includegraphics[width=0.9\linewidth]{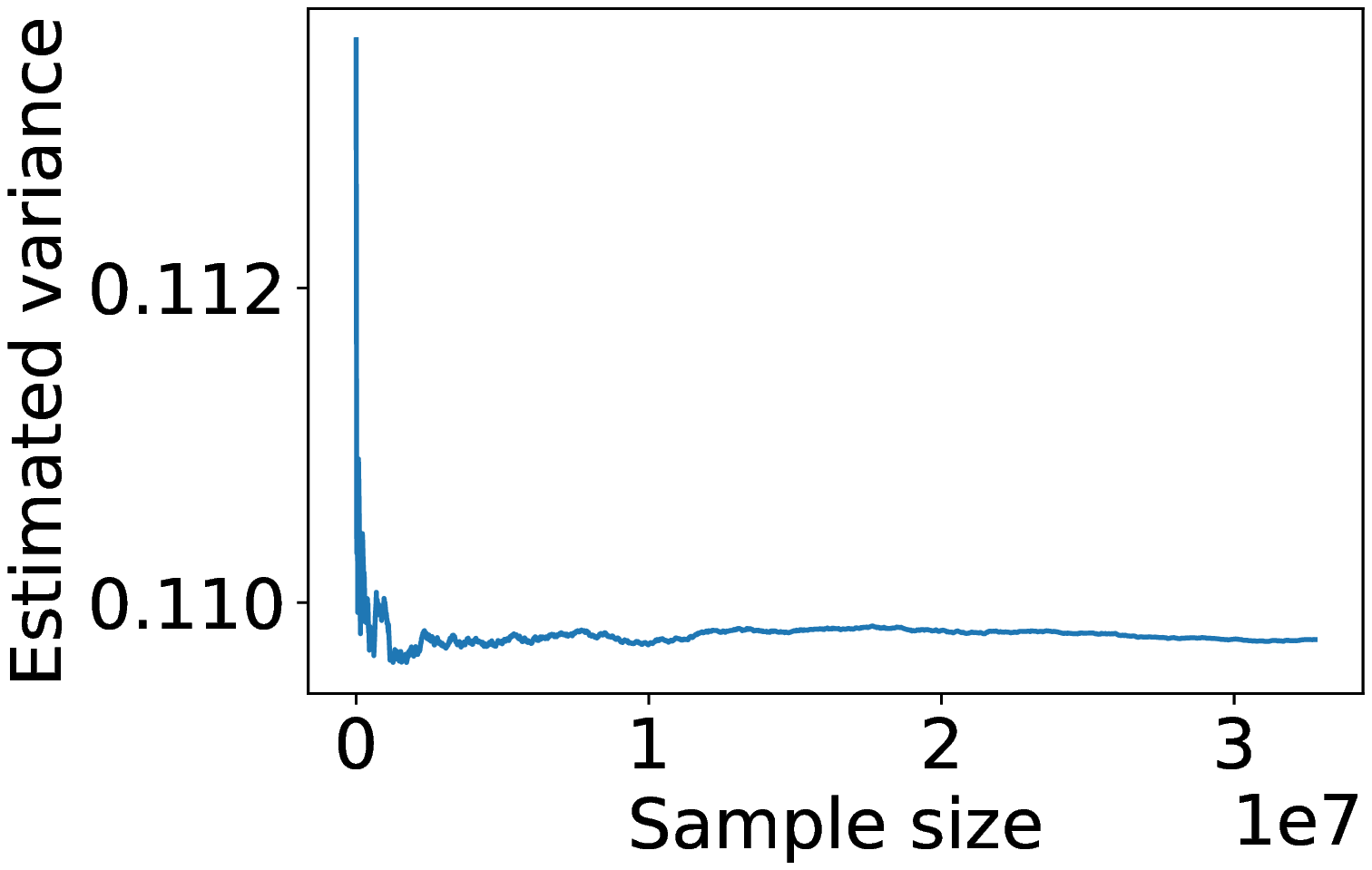}
	    \caption{ImageNet  variance}
	\end{subfigure}
	\\
		\begin{subfigure}{0.24\textwidth}
		\centering
		\includegraphics[width=0.85\linewidth]{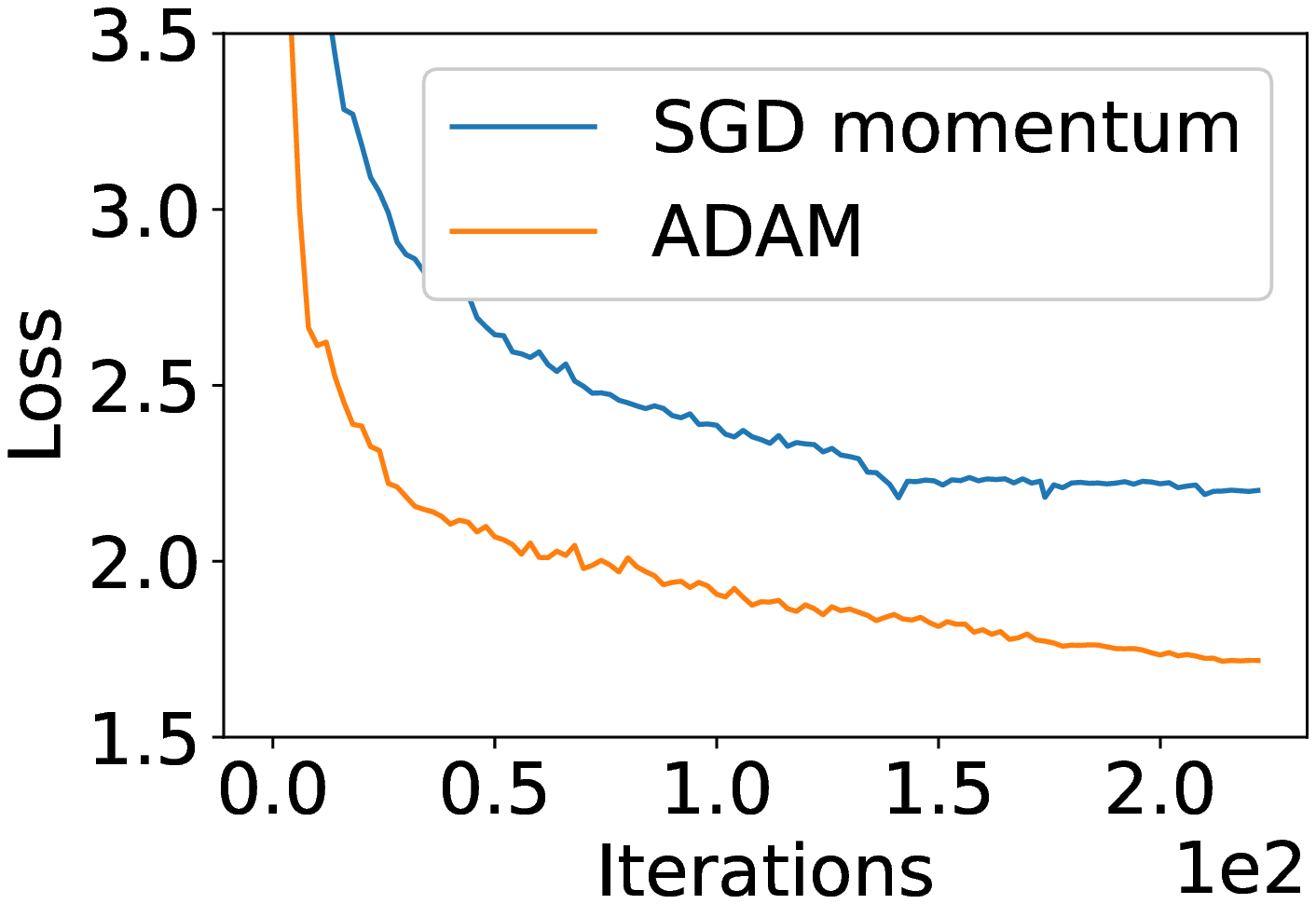}
		\caption{}\label{fig:loss-sgd-Adam}
	\end{subfigure}
    \begin{subfigure}{.24\textwidth}
	    \centering
    	\includegraphics[width=0.9\linewidth]{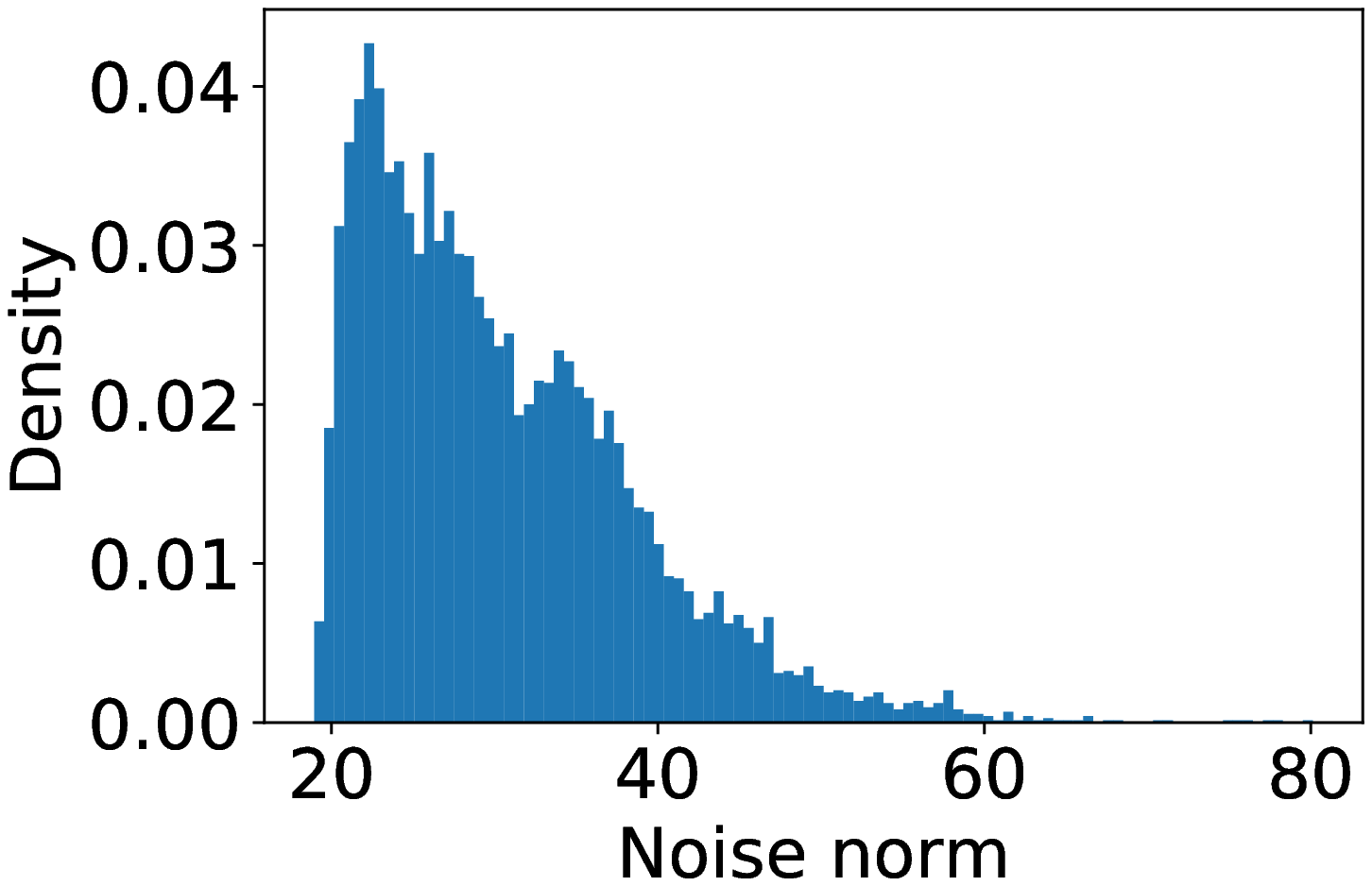}
	    \caption{Bert pretraining}
    \end{subfigure}
    \begin{subfigure}{.24\textwidth}
	    \centering
    	\includegraphics[width=0.9\linewidth]{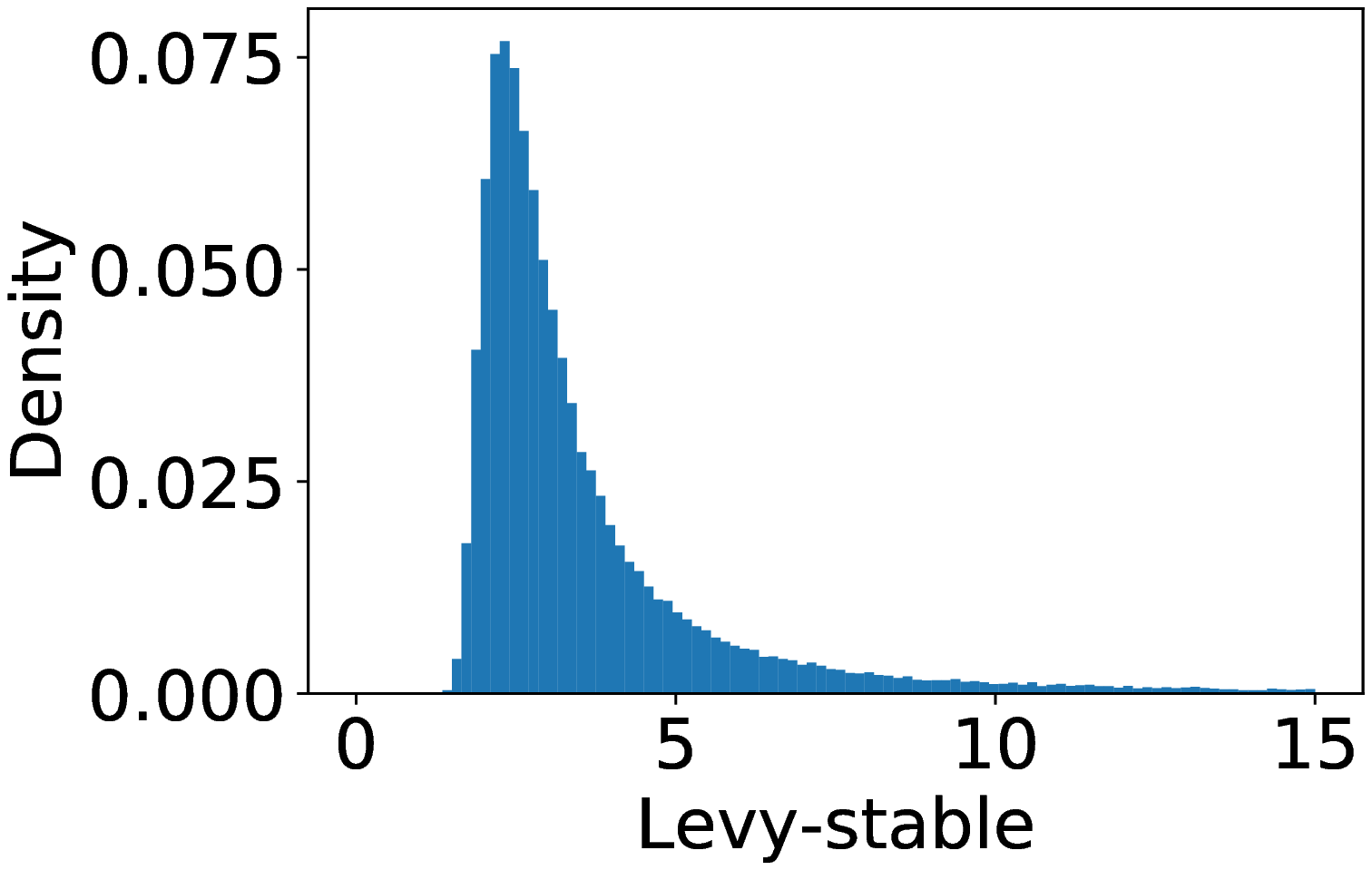}
	    \caption{Synthetic Levy-stable}
	\end{subfigure}
    \begin{subfigure}{.24\textwidth}
	    \centering
    	\includegraphics[width=0.9\linewidth]{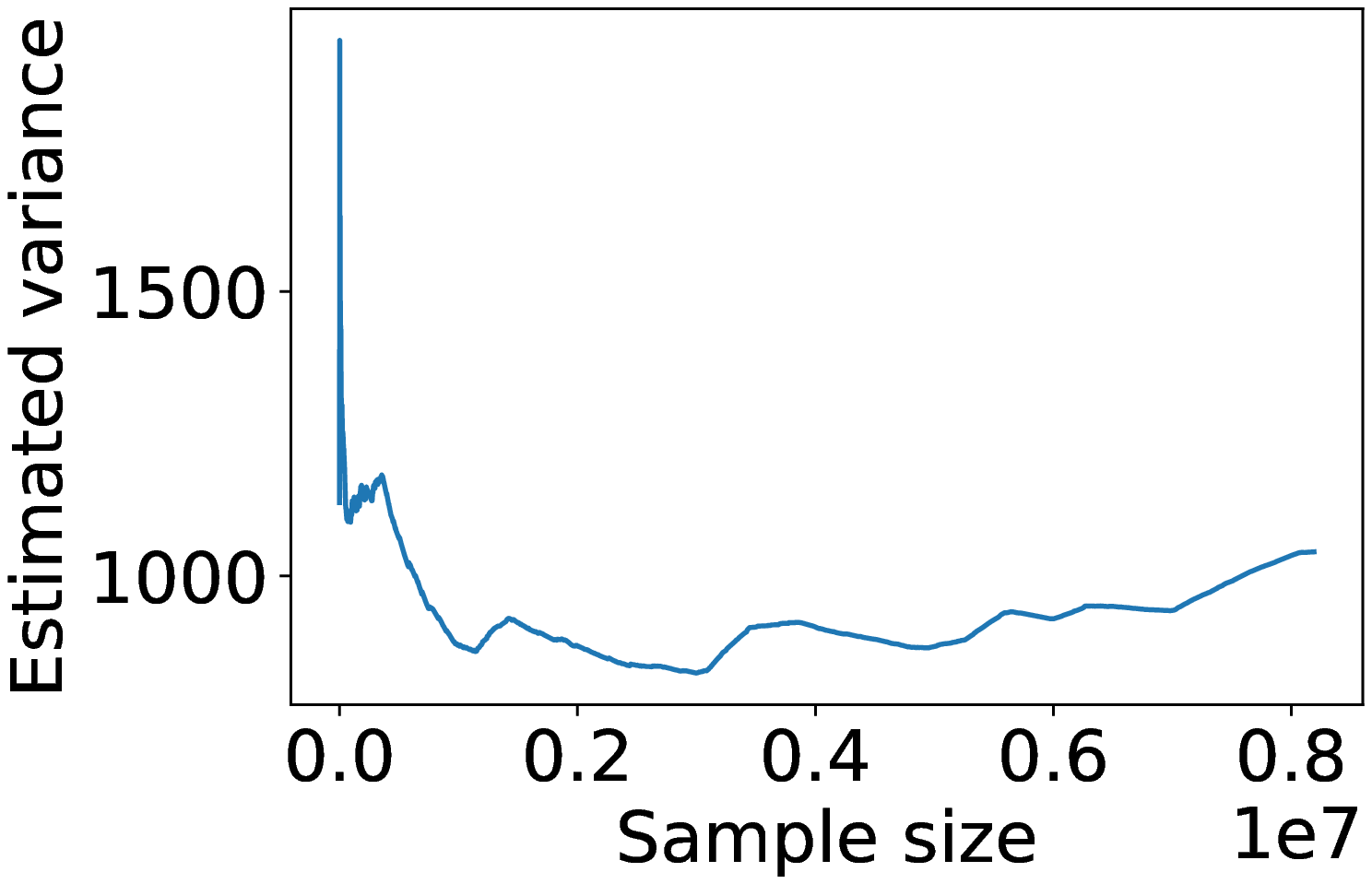}
        \caption{Bert  variance}
	\end{subfigure}
	\vspace{-0.2cm}
	\caption{(a) Validation loss for ResNet50 trained on ImageNet. SGD momentum outperforms Adam. (b) Histogram of sampled gradient noise for ResNet50 on Imagenet dataset. (c) Histogram of samples from a sum of squared Gaussians. (d) Estimated variance of the stochastic gradient for Resnet50. (e)Validation loss for BERT pretraining. Although hyperparameters for SGD are finetuned, a large performance gap is still observed between SGD and Adam.  (f) Histogram of sampled gradient nosie for BERT on Wikipedia+Books dataset. (g) Histogram of samples from a sum of squared $\alpha$-stable random variables. (h) Estimated variance of the stochastic gradient for BERT model.}\vspace{-0.2cm}
    \label{fig:heavy-noise}\vspace{-0.2cm}

\end{figure}





%% file: sections/main-results.tex

\vspace{-0.2cm}
\section{Convergence of gradient methods under heavy-tailed noise}\label{sec:thm}
\vspace{-0.2cm}

In this section we study the performance of SGD and adaptive methods under heavy-tailed noise. More precisely, we analyze algorithms of the following form
\begin{align}
	x_{k+1} = x_k - \eta_k g_k,
\end{align}
\vspace{-0.4cm}

where $x_k$ represent the current parameters, $\eta_k$ is the step size and $g_k$ is the stochastic (mini-batch) gradient evaluated at $x_k$. We show that if the stochasticity in the gradient $g_k$ is heavy-tailed, it is critical for the step sizes to be \textit{\textbf{adaptive}} i.e. $\eta_k$ must depend on the observed gradients. We propose to use one such algorithm GClip and prove that it obtains optimal convergence rates.
\vspace{-0.2cm}
\paragraph{Heavy-tailed noise.} 
Neural network training can be seen as minimizing a differentiable stochastic function $f(x) = \E_\xi [f(x, \xi)]$, where $f:\reals^d \to \reals$ can be potentially nonconvex and $\xi$ represent the mini-batches. At each iteration, we assume access to an \emph{unbiased} stochastic gradient $\E[g(x)] = \nabla f(x,\xi)$ corresponding to the parameters $x$, mini-batch $\xi$. We also need to bound how much noise is present in our stochastic gradients. In lieu of the usual bounded variance assumption, we use

\begin{assumption}[{\bf Bounded $\alpha-$moment}]
	\label{assump:alpha-moment}
	There exists positive real numbers $\alpha \in (1, 2]$ and $G > 0$ such that for all $x$,
	$\E[\|g(x)  - \nabla f(x) \|^\alpha] \le \sigma^\alpha$. We say noise is \textbf{heavy-tailed} if $\alpha < 2.$
\end{assumption}\vspace{-2mm}
The above assumption with $\alpha = 2$ corresponds to the standard variance bound, but in general is weaker. It is indeed possible (e.g. Pareto or $\alpha$-stable Levy random variables) for the variance of $g(x)$ to be unbounded, while simultaneously satisfying assumption~\ref{assump:alpha-moment} for $\alpha < 2$. One should note that even if the variance may not actually be infinite in practice, it might be too large to be practically useful. All our analyses and insights carry over to this setting as well.

The possibility that the variance is unbounded has a profound impact on the optimization process.
\begin{remark}[Nonconvergence of SGD]
 Consider the function $f(x) = x^2/2$ with noise satisfying $\E[\|g(x)  - \nabla f(x) \|^\alpha] = \sigma^\alpha$ for $\alpha < 2$, and  $\E[\|g(x)  - \nabla f(x) \|^2] = \infty$. Then, for any positive constants $\eta_k$ that do not depend on $g_k$, we have that $\E[\|\nabla f(x_k)\|^2] = \infty$.
\end{remark}
\vspace{-0.2cm}
\vspace{-0.2cm}
\begin{proof}
we denote the stochastic gradient $g_k := g(x_k) = \nabla f(x_k) + \xi_k = x_k + \xi_k$, where $\xi_k \in \mathbb{R}^d$ is a random variable with $\E\|\xi\|^2 = \infty, \E\|\xi\|^\alpha = \sigma^\alpha, \E[\xi] = \vec{0}.$  Then, $\E[\|\nabla f(x_{k+1})\|^2] = \E[\| x_{k+1} \|^2] = \E\|x_k - \eta_k g_k\|^2 = \E\|x_k - \eta_k (x_k + \xi)\|^2  = \E\|(1-\eta_k)x_k - \eta_k \xi \|^2 = \E\|(1-\eta_k)x_k\|^2 - 2 (1-\eta_k)\eta_k x_k^\top\E[\xi] + \eta_k^2 \E\| \xi \|^2 \geq \eta_k^2 \E\| \xi \|^2 = \infty. $ Note that this holds for \emph{any} fixed $\eta_k > 0$ even if allowed to depend on the statistics of the noise distribution (such as $\sigma$ or $\alpha$).
\end{proof}

The issue is that SGD is easily influenced by a single-stochastic gradient, which could be very large and incorrect. A simple strategy to circumvent this issue is to use a biased \emph{clipped} stochastic gradient estimator. This allows us to circumvent the problem of unbounded variance and ensures optimal convergence rates even under heavy-tailed noise. Our results are summarized in Table~\ref{tab:theory-summary}, and all proofs are relegated to the Appendices.

\begin{table}[]
\begin{center}
\caption{Error bounds ($f(x) - f^*$ for convex functions, $\|\nabla f(x)\|$ for nonconvex functions) after $k$ iterations: Define $\alpha$-moment as $\E[\|g(x) -  \nabla f(x)\|^\alpha] \leq \sigma^\alpha$ (Assump~\ref{assump:alpha-moment}) in the smooth nonconvex case and $\E[\|g(x) \|^\alpha] \leq G^\alpha$ (Assump~\ref{assump:alpha-moment-cvx}) in the strongly  case. In the standard setting ($\alpha =2$), GClip recovers the optimal rates. For heavy-tailed noise ($\alpha \in (1,2)$), GClip converges both for convex (Thm~\ref{thm:strongly-convex-convergence}) and non-convex functions (Thm~\ref{thm:nonconvex-convex-convergence}). We also show matching lower-bounds for all $\alpha \in (1,2]$ proving the optimality of clipping methods (Thm~\ref{thm:lower-bound}). 
}\vspace{0.2cm}
\begin{tabular}{@{}lllll@{}}
\toprule
                    & \multicolumn{2}{c}{Strongly Convex Function}   & \multicolumn{2}{c}{Non-Convex Function}  \\ \cmidrule(r){2-3}         \cmidrule(r){4-5}
     & Heavy-tailed noise                                                 & Standard noise                                  &Heavy-tailed noise                                                  & Standard noise                                     \\
    & ($\alpha \in (1,2)$)                                                & ($\alpha \geq 2$)                                  & ($\alpha \in (1,2)$)                                                  & ($\alpha \geq 2$)                                      \\\midrule
SGD                          & {N/A}                      & $\Oc\rbr[\big]{k^{-1}}$                                      & {N/A}                      & $\Oc\rbr[\big]{ k^{-\frac{1}{4}}}$                        \\
GClip                        & $\Oc\rbr[\big]{ k^{\frac{-(\alpha - 1)}{\alpha}}}$                                     & $\Oc\rbr[\big]{ k^{-1}}$                                      & $\Oc\rbr[\big]{k^{\frac{-(\alpha - 1)}{3\alpha -2}}}$                                     & $\Oc\rbr[\big]{ k^{-\frac{1}{4}}}$                                       \\
LowerBound                        & $\Omega\rbr[\big]{ k^{\frac{-(\alpha - 1)}{\alpha}}}$                                     & $\Omega\rbr[\big]{ k^{-1}}$                                      & $\Omega\rbr[\big]{k^{\frac{-(\alpha - 1)}{3\alpha -2}}}$                                     & $\Omega\rbr[\big]{ k^{-\frac{1}{4}}}$                                       \\
\bottomrule
\end{tabular}\label{tab:theory-summary}
\end{center}\vspace{-0.2cm}\vspace{-0.2cm}

\end{table}

\vspace{-0.2cm}
\subsection{Convergence of  Clipped Methods}\label{sec:gclip-rate}
A simple clipping strategy is to globally clip the norm of the update to threshold $\tau_k$:
\vspace{-0.1cm}
\begin{align}
x_{k+1} &=x_k - \eta_k \min\cbr[\big]{\tfrac{\tau_k}{\|g_k\|}, 1}g_k\,,\ \tau_k \in \mathbb{R}_{\geq 0}  \label{eqn:g-clip}\tag{GClip}
\end{align}
We refer to this strategy as GClip (Global Clip), as opposed to coordinate-wise clipping which we discuss later. We first state the rates for smooth non-convex functions.

\begin{theorem}[{\bf Non-convex convergence}]\label{thm:nonconvex-convex-convergence}
	Suppose that $f$ is $L$-smooth and that the stochastic gradients satisfy Assumption~\ref{assump:alpha-moment} for $\alpha \in (1,2]$. Let $\{x_k\}$ be the iterates of GClip with parameters $\eta_k = \eta = \min\{\frac{1}{4L}, \frac{\sigma^\alpha}{L\tau^\alpha} , \frac{1}{24L\tau}\}$ and $\tau_k = \tau = \max\{2, 48^{1/(\alpha-1)}\sigma^{\alpha/(\alpha - 1)},  8\sigma, \rbr[\big]{\frac{f_0}{\sigma^2K}}^{\frac{\alpha}{3\alpha - 2}} \big/ L^{\frac{2\alpha - 2}{3\alpha - 2}}, \} $. Then for $F_0 := f(x_0) - f^*$,
	\vspace{-0.2cm}
	\begin{align*}
	\frac{1}{K}\sum_{k=1}^K &\E[\min\{\norm{\nabla f(x_{k})}, \norm{\nabla f(x_{k})}^2\}] =\mathcal{O}(K^{\frac{-2\alpha+2}{3\alpha - 1}}),.
	\end{align*}
\end{theorem}
\begin{remark}
When $\|\nabla f(x_k)\| \le \epsilon \ll 1$,  $\norm{\nabla f(x_{k})}^2 \ll \norm{\nabla f(x_{k})}$. Hence the dominant term on the left hand side of the inequality above is $\norm{\nabla f(x_{k})}^2$. The right hand side is easily observed to be $\Oc(K^{^{-\frac{2(\alpha-1)}{3\alpha - 2}}})$. Together, this implies a convergence rate of $\E \|\nabla f(x)\| \leq \Oc(K^{^{-\frac{(\alpha-1)}{3\alpha - 2}}})$.
\end{remark}
We prove improved rates of convergence for non-smooth \emph{strongly-convex} functions in a bounded domain. Due to limited space, we relegate the definitions and assumptions to Appendix~\ref{sec:app-assump}.
\begin{theorem}[{\bf Strongly-convex convergence}]\label{thm:strongly-convex-convergence}
	Suppose that  the stochastic gradients satisfy Assumption~\ref{assump:alpha-moment-cvx} for $\alpha \in (1,2]$. Let $\{x_k\}$ be the iterates of projected GClip~\eqref{eqn:proj-g-clip} with clipping parameter $\tau_k = G k^{\alpha -1}$ and steps-size $\eta_k = \frac{4}{\mu (k+1)}$. Define the output to be a $k$-weighted combination of the iterates:   
	$
	\bar x_{k} = \sum_{j=1}^{k} j x_{j-1} / (\sum_{j=1}^{k} j)\,.
	$
	Then the output $\bar x_k$ satisfies:
	\[
	\E[f(\bar x_k)] - f(x^\star) \leq \tfrac{16 G^2}{\mu(k+1)^{2(\alpha-1)/\alpha}}\,.
	\]
\end{theorem}

The rates of convergence for the strongly convex and non-convex cases in Theorem~\ref{thm:strongly-convex-convergence} and Theorem \ref{thm:nonconvex-convex-convergence} exactly match those of the usual SGD rates ($\Oc(1/\sqrt{k})$ for convex and $\Oc(k^{-\frac{1}{4}})$ for non-convex) when $\alpha = 2$ and gracefully degrade for $\alpha \in (1,2]$. As we will next show, both the strongly convex rates and non-convex rates of GClip are in fact optimal for every $\alpha \in (1,2]$. 

\vspace{-0.2cm}
\subsection{Theoretic lower bounds}\label{sec:lower}
\vspace{-0.2cm}
We prove that the rates obtained with GClip are optimal up to constants. First, we show a strong lower-bound for the class of convex functions with stochastic gradients satisfying $\E[|g(x)|^\alpha] \leq 1$. 
This matches the upper bounds of Theorems~\ref{thm:strongly-convex-convergence} and~\ref{thm:c-clip-convergence} for strongly-convex functions, showing that the simple clipping mechanism of GClip is (up to constants) information theoretically optimal, providing a strong justification for its use.
\begin{theorem}\label{thm:lower-bound}
	For any $\alpha \in (1,2]$ and any (possibly randomized) algorithm $\Ac$, there exists a problem $f$ which is 1-strongly convex and 1-smooth ($\mu =1$ and $L=1$), and stochastic gradients which satisfy Assumptions~\ref{assump:alpha-moment-cvx} with $G \leq 1$ such that the output $x_k$ of the algorithm $\Ac$ after processing $k$ stochastic gradients has an error
	\[
	\E[f(x_k)] - f(x^\star) \geq \Omega\bigl(\tfrac{1}{k^{2(\alpha -1)/\alpha}}\bigr)\,.
	\]
\end{theorem}
Next, we examine non-convex functions.
\begin{theorem}\label{thm:lower-bound-nonconvex}
	Given any $\alpha \in (1,2]$, smoothness constant $L$, and (possibly randomized) algorithm $\Ac$, there exists a constant $c_1$ and an $L$-smooth function $f$ with stochastic gradients satisfying Assumption~\ref{assump:alpha-moment} for any given $\sigma \ge c_1 \sqrt{(f(0) - f^*)L}$ such that the output $x_k$ of the algorithm $\Ac$ after processing $k$ stochastic gradients has an error
	\begin{align*}
	\E[\|\nabla f(x_k)\|] \geq \Omega\bigl(\tfrac{1}{k^{(\alpha - 1)/(3\alpha - 2)}}\bigr)\,.
	\end{align*}
\end{theorem}
Theorem \ref{thm:lower-bound-nonconvex}, proven in Appendix \ref{sec:non-convex-lower-proof}, extends the recent work of \citep[Theorem 1]{arjevani2019lower} to heavy-tailed noise. Here, the lower-bound matches the upper-bound in Theorem~\ref{thm:nonconvex-convex-convergence} up to constants, proving its optimality.

%% file: sections/experiments.tex
\vspace{-0.2cm}

\section{Faster Optimization with Adaptive Coordinate-wise Clipping}\label{sec:transformerxl}
\vspace{-0.2cm}


The previous section showed that adaptive step sizes (which depend on the gradients) are essential for convergence under heavy-tailed noise, and also showed that GClip provides the optimal rates.
There are of course other adaptive methods such as Adam which employs not only the current gradients but also \emph{all} past gradients to adaptively set \emph{coordinate-wise} step-sizes. In this section, we study why coordinate-wise clipping may yield even faster convergence than GClip, and show how to modify GClip to design an Adaptive Coordinate-wise Clipping algorithm (ACClip).


\subsection{Coordinate-wise clipping}
The first technique we use is applying coordinate-wise clipping instead of global clipping. We had previously assumed a global bound on the $\alpha$-moment of the \emph{norm} (or variance) of the stochastic gradient is bounded by $\sigma$. However, $\sigma$ might be hiding some dimension dependence $d$. We show a more fine-grained model of the noise in order to tease out this dependence.

\begin{assumption}[{\bf Coordinate-wise $\alpha$ moment}]\label{assump:coord-noise} 
	Denote $\{g_i(x)\}$ to be the coordinate-wise stochastic gradients for $i \in [d]$. We assume there exist constants $\{B_i\} \geq 0$ and $\alpha \in (1,2]$ such that $
	\E\sbr*{\abs*{g_i(x)}^{\alpha}} \leq {B}^{\alpha}_i \,.$
\end{assumption}\vspace{-0.2cm}

For the sake of convenience, we denote $B = [B_1;B_2;\cdots;B_d] \in \mathbb{R}^d$, $\norm{B}_a = \rbr{\sum B_i^a}^{1/a}$. Under this more refined assumption, we can show the following corollary:
\begin{corollary}[{\bf GClip under coordinate-wise noise}]\label{cor:coord-g-clip-convergence}
	Suppose we run GClip under Assumption      ~\ref{assump:coord-noise} to obtain the sequence $\{x_k\}$. If $f$ is $\mu$-strongly convex, with appropriate step-sizes and averaging, the output $\bar x_k$ satisfies
	\vspace{-4mm}
	\begin{align*}
	    	\quad \quad \E[f(\bar x_k)] - f(x^\star) \leq \tfrac{16 d \norm{B}^2_\alpha}{\mu(k+1)^{2(\alpha-1)/\alpha}}\,.
	\end{align*}
	\vspace{-6mm}
\end{corollary}
	\vspace{-2mm}
Thus, the convergence of GClip can have a strong dependence on $d$, which for large-scale problems might be problematic. We show next that using coordinate-wise clipping removes this dependency:
	\vspace{-2mm}
\begin{align}
x_{k+1} &=x_k - \eta_k \min\cbr[\big]{\tfrac{\tau_k}{|g_k|}, 1}  g_k\,, \text{ } \tau_k \in \mathbb{R}_{\geq 0}^d\,. \label{eqn:c-clip}\tag{CClip}
\end{align}

\begin{theorem}[{\bf CClip under coordinate-wise noise}]\label{thm:c-clip-convergence}
	Suppose we run CClip under the Assumption of~\ref{assump:coord-noise} with $\tau_k = B k^{\alpha -1}$ to obtain the sequence $\{x_k\}$. Then, if $f$ is $\mu$-strongly convex, with appropriate step-sizes and averaging, the output $\bar x_k$ satisfies
	\vspace{-0.2cm}
	\[
	\E[f(\bar x_k)] - f(x^\star) \leq \tfrac{16 \norm{B}^2_2}{\mu(k+1)^{2(\alpha-1)/\alpha}}\,.
	\]
	\vspace{-0.4cm}
\end{theorem}
	\vspace{-2mm}
Note that $\norm{B}_2 \le \norm{B}_\alpha$. CClip has a worst-case convergence independent of $d$ under the coordinate-wise noise model.  Similar comparison between GClip and CClip can be done for non-convex conditions too, but we skip for conciseness. Though we only compare upper-bounds here, when the noise across coordinates is independent the upper bounds may be tight (see Lemma~\ref{lemma:g-clip}). 
	\vspace{-2mm}
\subsection{Online moment estimation}

\begin{figure*}[t!]
	\begin{subfigure}{\textwidth}
		\centering
		\includegraphics[width=0.85\linewidth]{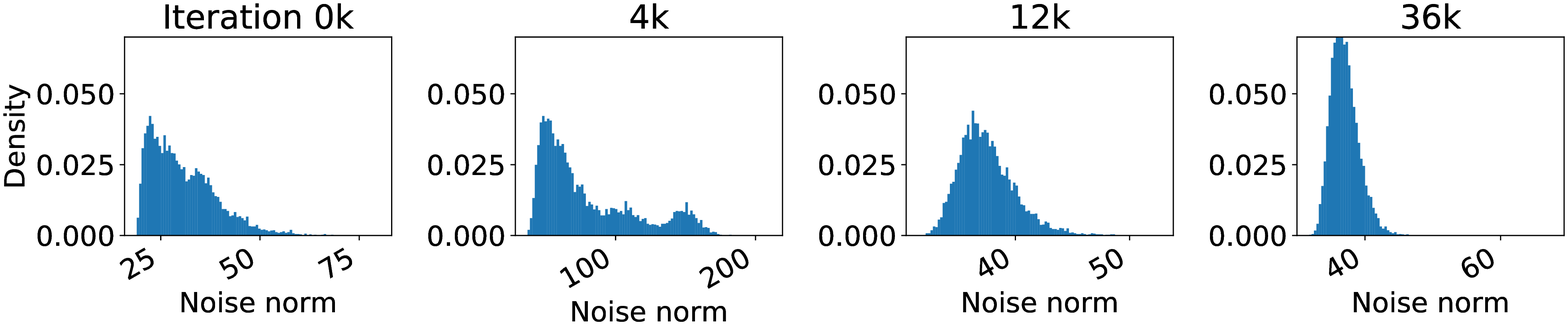}

		\caption{Development of noise distribution during BERT training.}
	\end{subfigure}

	\begin{subfigure}{\textwidth}
		\centering
		\includegraphics[width=0.85\linewidth]{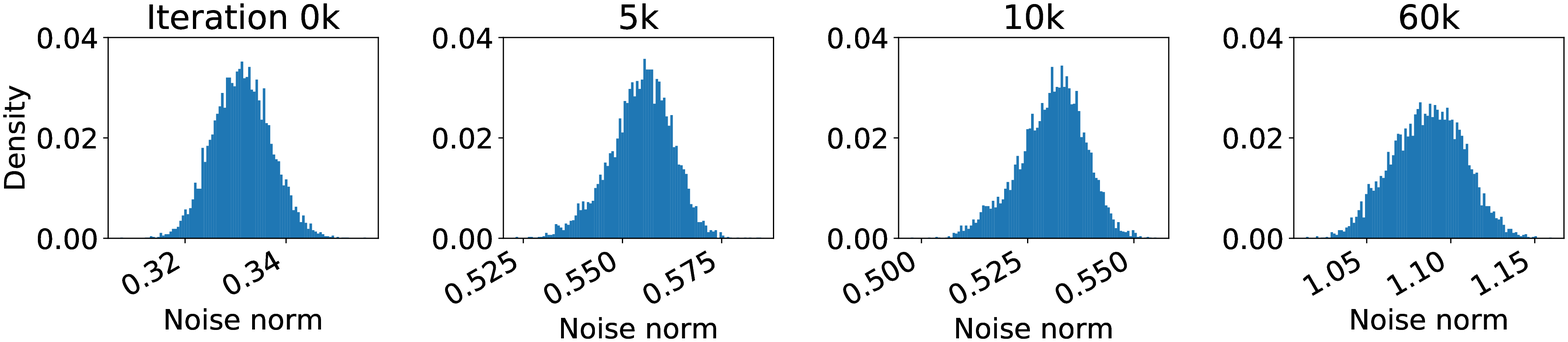}

		\caption{Development of noise distribution during ResNet50 training on ImageNet.}
	\end{subfigure}

	\caption{The distribution of gradient noise is non-stationary during BERT training, while it remains almost unchanged for ResNet training on ImageNet.}
	\label{fig:nonstationary}
\end{figure*}

We now present the second technique that is motivated by our observation in Figure~\ref{fig:nonstationary}. There, the distribution of gradient noise at the beginning of different epochs is shown during training for BERT with Wikipedia (top) as well as ResNet  with ImageNet (bottom). The result highlights that the noise distribution is not only heavy-tailed, but also non-stationary during BERT training and becomes increasingly more concentrated. In contrast, for the ResNet model the noise distribution remains mostly unchanged. 

Since the scale of the noise changes drastically during training for BERT model and our theoretical analysis suggest that we should clip proportional to the noise level, we propose to use an exponential moving average estimator to estimate the moment and clip the gradient accordingly (line 4,5 of Alg~\ref{alg:acclipp}). This, combined with the momentum term leads to our proposed ACClip algorithm in Algorithm~\ref{alg:acclipp}. On a high level, the algorithm applies clipping to the momentum term, where the clipping threshold is proportional to the estimated moment using an exponential moving average. From our experiment, we found the conservative choice of $\alpha = 1$ leads to the best performance.

\begin{algorithm}[t] 
	\caption{ACClip}\label{alg:acclipp}
	\begin{algorithmic}[1]
		\State $x, m_k \gets x_0, 0$
		\For {$k = 1, \cdot, T$}
		\State $m_k \gets \beta_1 m_{k-1} + (1-\beta_1) g_k $
		\State $\tau_k^\alpha \gets \beta_2 \tau_{k-1}^\alpha + (1 - \beta_2) |g_k|^\alpha$
		\State $\hat{g}_k \gets \min\cbr[\big]{\tfrac{\tau_k}{|m_k|+\epsilon}, 1}m_k$
		\State $x_k \gets x_{k-1} - \eta_k \hat{g}_k$
		
		\EndFor
		\Return $x_K$, where random variable $K$ is supported on $\{1, \cdots, T\}$.
	\end{algorithmic}
\end{algorithm}

\vspace{-2mm}
	
\section{Experiments}\label{sec:experiment}
\vspace{-2mm}

In this section, we first verify the effect of coordinate-wise clipping and moment estimation introduced in Section~\ref{sec:transformerxl}. We then perform extensive evaluations of ACClip on BERT pre-training and fine-tuning tasks and demonstrate its advantage over Adam in Section~\ref{sec:acclip-exp}. For completeness, an experiment on ImageNet is included in Appendix~\ref{sec:renet-exp}. Finally, we start with a few more experiments on the noise distribution in neural network training.

\subsection{From GClip to ACClip}\label{sec:transformerxl-exp}
	
In this section we instantiate the argument in Section~\ref{sec:transformerxl} with a set of experiments. As seen in Figure~\ref{fig:transformerxl}, global clipping improves the vanilla SGD algorithm but is still far from the ADAM baseline. We apply two techniques (coordinate-wise clipping and online moment estimation) onto the clipped SGD algorithm analyzed in Section~\ref{sec:thm}. We use a set of experiments on Transformer-XL training to demonstrate the effect of each technique.

\begin{figure}
	\begin{subfigure}{.33\textwidth}
		\centering
		\includegraphics[width=0.99\linewidth]{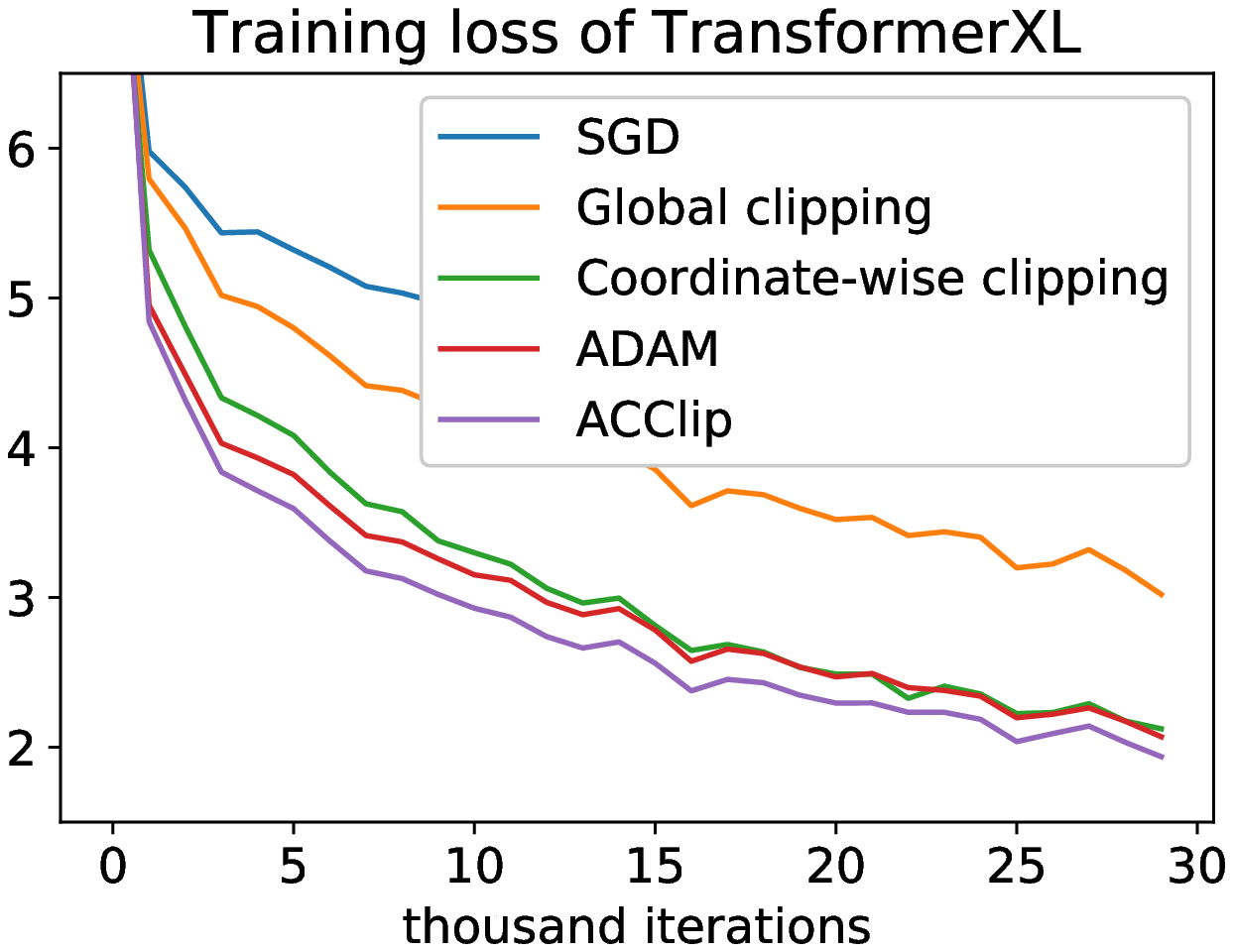}
		\vspace{-0.2cm}
		\captionsetup{width=.99\linewidth}
		\caption{}	\label{fig:transformerxl}
	\end{subfigure}
		\begin{subfigure}{.33\textwidth}
		\centering
		\includegraphics[width=0.99\linewidth]{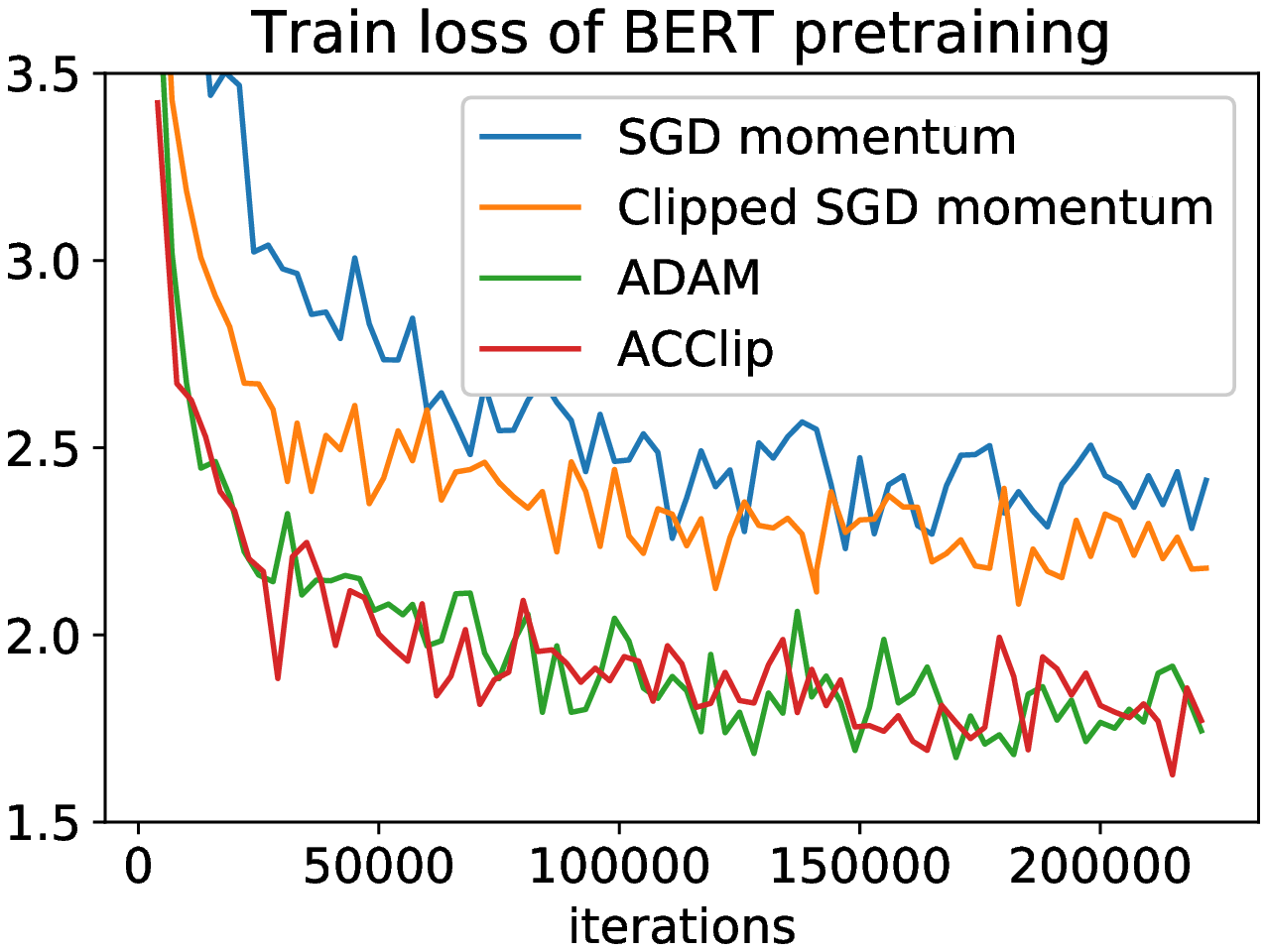}
		\vspace{-0.2cm}
		\captionsetup{width=.99\linewidth}
		\caption{}	\label{fig:transformerxl}
	\end{subfigure}
	\begin{subfigure}{.33\textwidth}
		\centering
		\includegraphics[width=\linewidth]{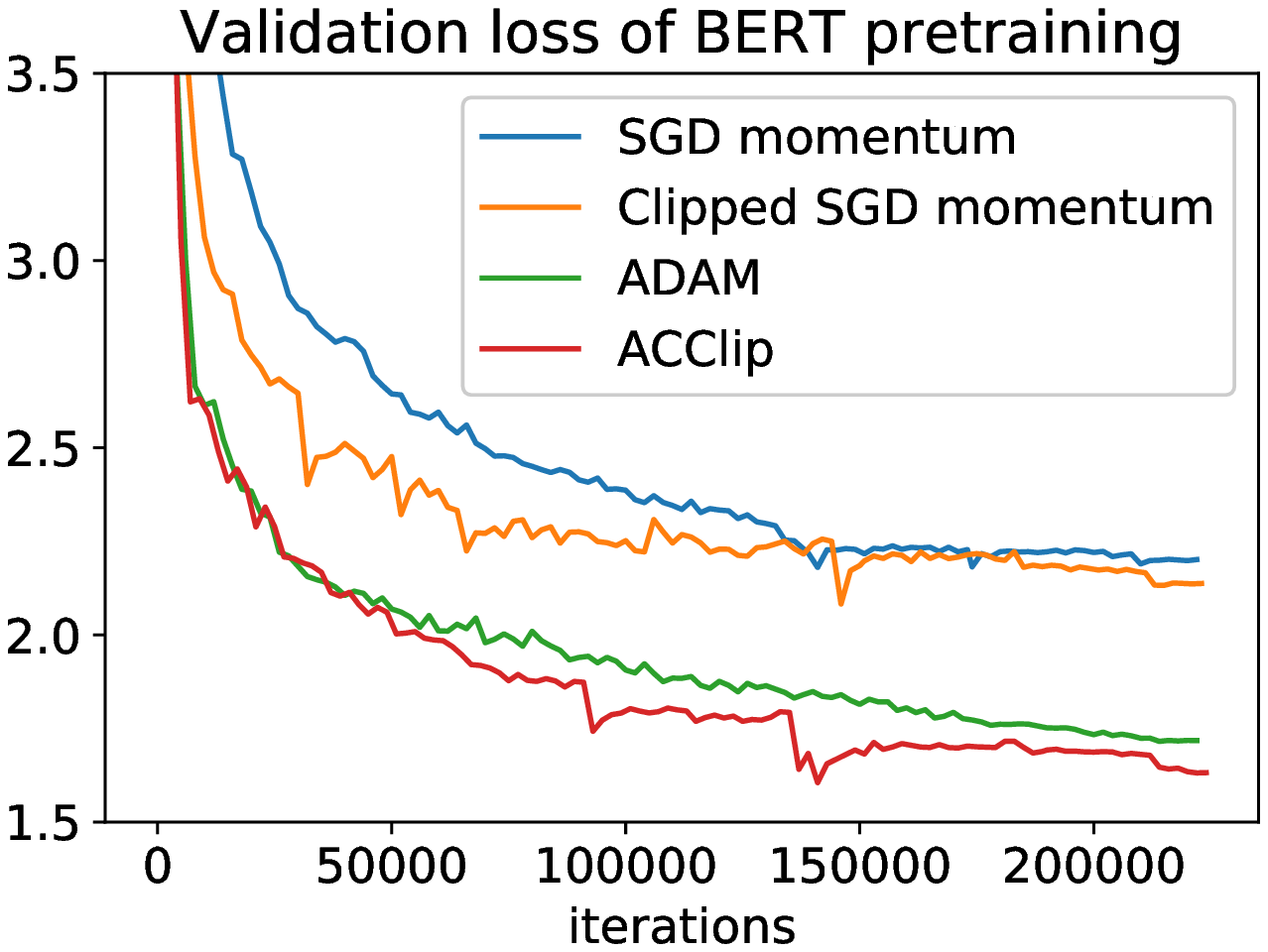}
		\vspace{-0.2cm}
		\captionsetup{width=.99\linewidth}
		\caption{}\label{fig:loss-curves-bert}
	\end{subfigure}
	\caption{(a) Performance of different algorithms for training a toy transformer-XL model described in Section~\ref{sec:transformerxl}. (b) Train and (c) validation loss for BERT$_{base}$ pretraining with the sequence length of 128. While there remains a large gap between non-adaptive methods and adaptive methods, clipped SGD momentum achieves faster convergence compared to standard SGD momentum. The proposed algorithm for adaptive coordinate-wise clipping (ACClip) achieves a lower loss than Adam.}
\end{figure}

\paragraph{Experiment setup}
We train  a 6-layer Transformer-XL model\cite{dai2019transformer} on PTB dataset as a proof of concept. Our main experiments will be on BERT pretraining and finetuning described in the next subsection~\ref{sec:acclip-exp}. We use adapt the author's github repo\footnote{https://github.com/kimiyoung/transformer-xl/tree/master/pytorch}, and replace the number of layers of the base model by 6. We then select the PTB data as input and set the maximum target length to be 128. The results are shown in Figure~\ref{fig:transformerxl}. 

\paragraph{Observations}  From Figure~\ref{fig:transformerxl}, we can tell that global clipping (orange curve) indeed speeds up vanilla SGD but is still much worse compared to the ADAM baseline provided by the code base. After replacing global clipping with coordinate-wise clipping, we see that the performance is already comparable to the ADAM baseline. Finally, after using the moment estimation to determine the clipping threshold, we are able to achieve faster convergence than ADAM.

\subsection{Performance of ACClip for BERT pre-training and fine-tuning}\label{sec:acclip-exp}


We now evaluate the empirical performance of our proposed ACClip algorithm on BERT pre-training as well fine-tuning using the SQUAD v1.1 dataset. As a baseline, we use Adam optimizer and the same training setup as in the BERT paper \citep{devlin2018bert}. For ACClip, we set $\tau = 1, \textit{learning rate} = \text{1e-4}, \beta_1 = 0.9, \beta_2 = 0.99, \epsilon=\text{1e-5}$ and $\textit{weight decay} = \text{1e-5}$. We compare both setups on BERT models of three different sizes, BERT\textsubscript{base} with 6 and 12 layers as well as BERT\textsubscript{large} with 24 layers.

Figure~\ref{fig:transformerxl} and~\ref{fig:loss-curves-bert} shows the loss for pretraining BERT\textsubscript{base} using SGD with momentum, GClip, Adam and ACClip. The learning rates and hyperparameters for each method have been extensively tuned to provide best performance on validation set. However, even after extensive tuning, there remains a large gap between (clipped) SGD momentum and adaptive methods. Furthermore, clipped SGD achieves faster convergence as well as lower final loss compared to standard SGD. Lastly, the proposed optimizer ACClip achieves a lower loss than the Adam. Table~\ref{fig:adam-vs-acclip-pretrain} further shows that ACClip achieves lower loss and higher masked-LM accuracy for all model sizes.

Next, we evaluate ACClip on the SQUAD v1.1 fine-tuning task. We again follow the procedure outlined in \citep{devlin2018bert} and present the results on the Dev set in Table~\ref{fig:adam-vs-acclip-squad}. Both for F1 as well as for exact match, the proposed algorithm outperforms Adam on all model sizes.
The experimental results on BERT pretraining and fine-tuning indicate the effectiveness of the proposed algorithm. 
\begin{table}[t]
	\fontsize{9}{10.5}
	\centering
	\caption{\label{fig:adam-vs-acclip-pretrain}
		\textbf{BERT pretraining: Adam vs ACClip.} Compared to Adam, the proposed ACClip algorithm achieves better evaluation loss and Masked LM accuracy for all model sizes.}
	\begin{tabular}{@{}lcccccc@{}} \toprule
		& \multicolumn{2}{c}{\textbf{  BERT Base 6 layers  }} &  \multicolumn{2}{c}{\textbf{  BERT Base 12 layers  }} & \multicolumn{2}{c}{\textbf{  BERT Large 24 layers}} \\ 
		&   $\text{Val. loss}$ &  $\text{Accuracy}$    &  $\text{Val. loss}$ & $\text{Accuracy}$    &  $\text{Val. loss}$ &  $\text{Accuracy}$ \\
		\midrule 
		$\text{Adam}$ &  $1.907$ & $63.45$    &  $1.718$ & $66.44$    &  $1.432$ & $70.56$ \\ 
		$\text{ACClip}$  &   \textbf{1.877} & \textbf{63.85}   & \textbf{1.615} & \textbf{67.16}    & \textbf{1.413} & \textbf{70.97}  \\ 
		\bottomrule
	\end{tabular}
		\vspace{-0.3cm}
\end{table}
\begin{table}[t]
\fontsize{9}{10.5}
    \centering
    \caption{\label{fig:adam-vs-acclip-squad}
	  \textbf{SQUAD v1.1 dev set: Adam vs ACClip}. The mean and standard deviation of F1 and exact match score for 5 runs. The first row contains results reported from the original BERT paper, which are obtained by picking the best ones out of 10 repeated experiments.}
	\begin{tabular}{@{}lcccccc@{}} \toprule
	& \multicolumn{2}{c}{\textbf{  BERT Base 6 layers  }} &  \multicolumn{2}{c}{\textbf{  BERT Base 12 layers  }} & \multicolumn{2}{c}{\textbf{  BERT Large 24 layers}} \\ 
	&  $\text{EM}$  &  $\text{F1}$     &  $\text{EM}$  &  $\text{F1}$      &  $\text{EM}$ &  $\text{F1}$   \\
	\midrule 
	\multicolumn{2}{@{}l}{{$\text{Adam}$}{\fontsize{7}{10.5} $\text{(Devlin et al., 2018) }$}}   &    &  $80.8$ & $88.5$    &   $84.1$ & $90.9$ \\
	$\text{Adam}$  &  $76.85\pm0.34$ & $84.79\pm0.33$    &  $81.42 \pm 0.16$ & $88.61\pm0.11$   &  $83.94\pm0.19$ & $90.87\pm0.12$ \\ 
	$\text{ACClip}$   &   \textbf{78.07 $\pm$ 0.24} & \textbf{85.87 $\pm$ 0.13}   & \textbf{81.62 $\pm$ 0.18} & \textbf{88.82  $\pm$ 0.10}    & \textbf{84.93 $\pm$ 0.29} & \textbf{91.40 $\pm$ 0.15}  \\ 
	\bottomrule
    \end{tabular}
\end{table}

\vspace{-0.2cm}
\subsection{Noise Patterns in BERT and ImageNet Training}\label{sec:noise-exp}
\vspace{-0.2cm}


\begin{figure*}[t!]
	\centering
	\begin{subfigure}{.24\textwidth}
		\centering
		\includegraphics[width=0.95\linewidth]{figs/bert_wiki_0k.eps}
		\captionsetup{width=.99\linewidth}
		\caption{Attention + Wikipedia}
	\end{subfigure}
	\begin{subfigure}{.24\textwidth}
		\centering
		\includegraphics[width=0.95\linewidth]{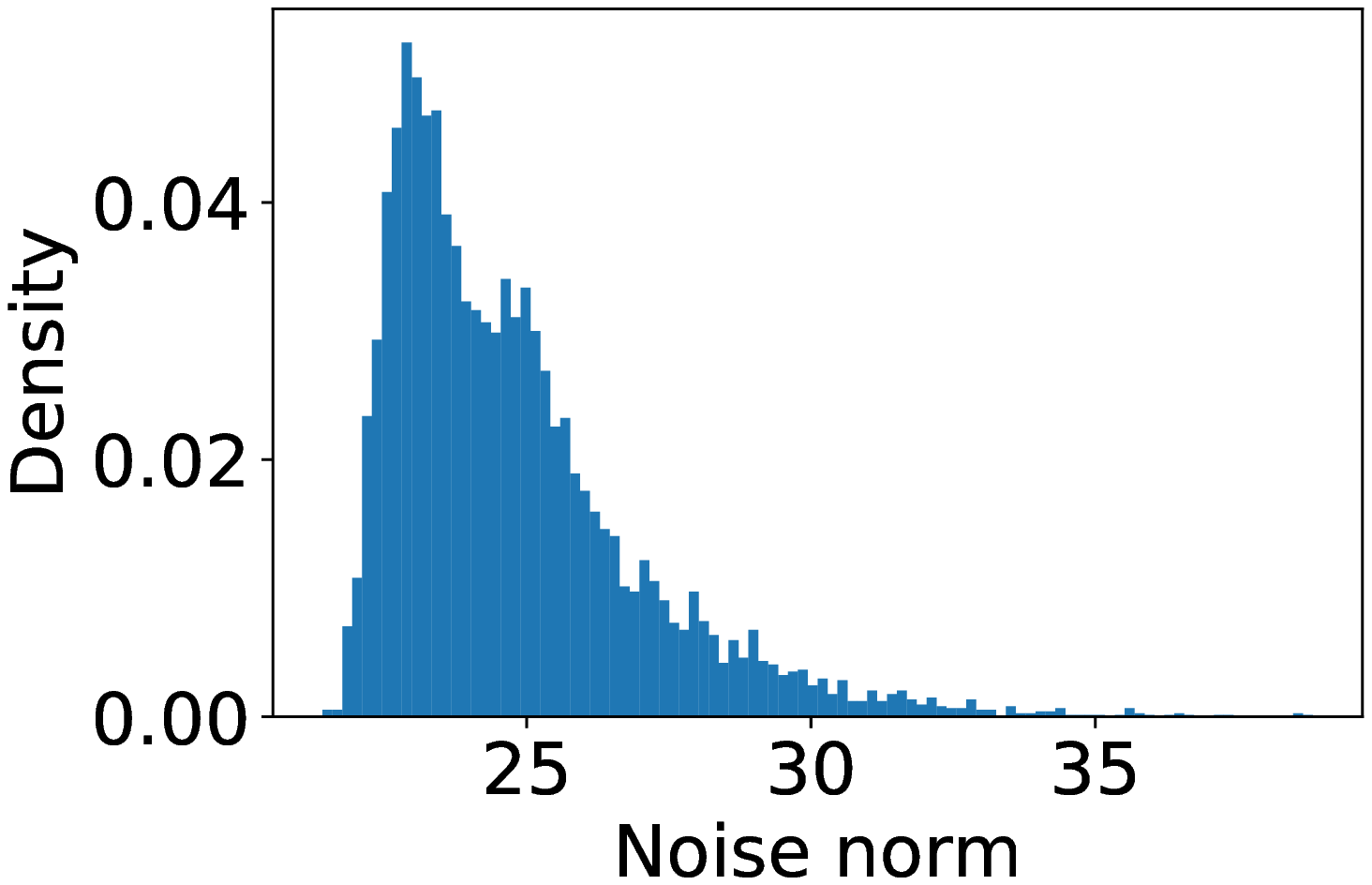}
		\captionsetup{width=.99\linewidth}
		\caption{Attention + Gaussian}
	\end{subfigure}
	\begin{subfigure}{.24\textwidth}
		\centering
		\captionsetup{width=.9\linewidth}
		\includegraphics[width=0.95\linewidth]{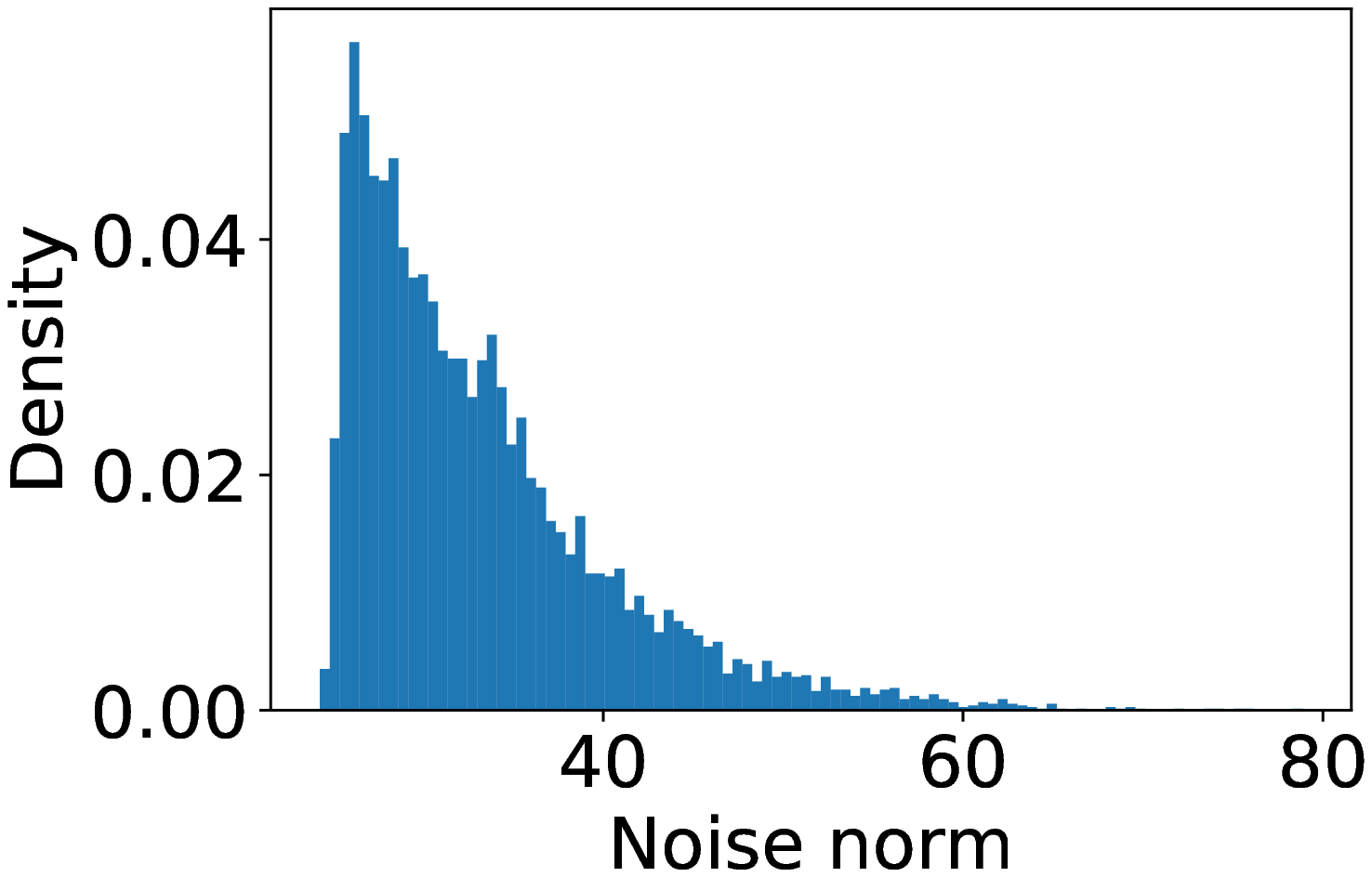}
		\caption{Resnet + Wikipedia}
	\end{subfigure}
	\begin{subfigure}{.24\textwidth}
		\centering
		\captionsetup{width=.9\linewidth}
		\includegraphics[width=0.95\linewidth]{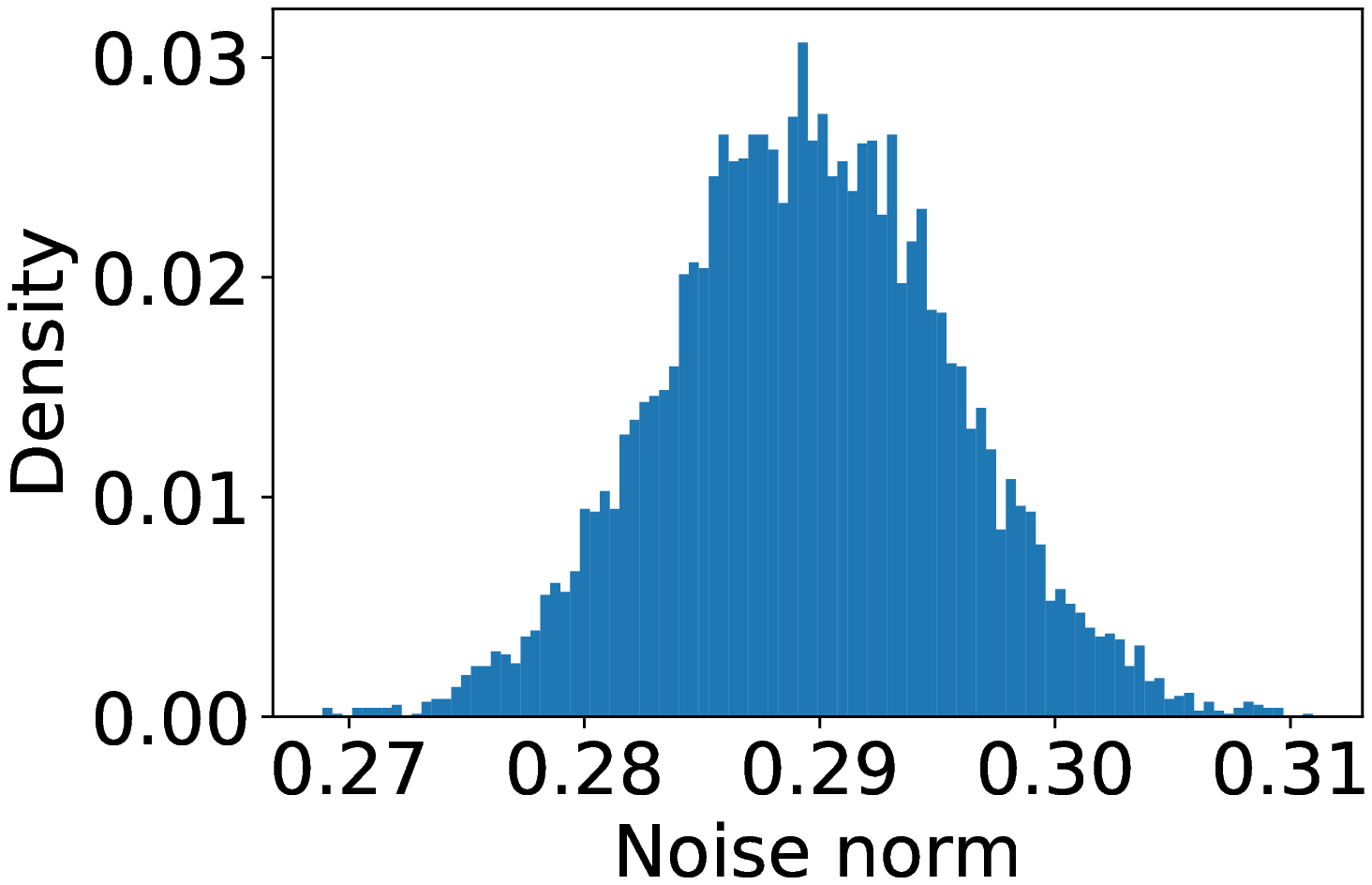}
		\caption{Resnet + Gaussian.}
	\end{subfigure}
	\caption{Distribution of gradient noise norm in Attention and ResNet models on two data sources: Wikipedia and synthetic Gaussian. The heavy-tailed noise pattern results from the interaction of both model architecture as well as data distribution.}
	\label{fig:data-vs-arch}
\end{figure*}

In our initial analysis in Figure~\ref{fig:heavy-noise}, we observe that training an attention model on Wikipedia leads to heavy-tailed noise whereas training a ResNet on ImageNet data leads to well-concentrated noise.
Here, we aim to disentangle the effect that model architecture and training data have on the shape of gradient noise.  
To this end, we measure the distribution of the gradient noise norm in an Attention and a ResNet model on both Wikipedia and synthetic Gaussian data. We used BERT\textsubscript{base} as the Attention model, and the ResNet is constructed by removing the self-attention modules within the transformer blocks. Gaussian synthetic data is generated by replacing the token embedding layer with normalized Gaussian input. 
The resulting noise histograms are shown in Figure~\ref{fig:data-vs-arch}. The figure shows that the Attention model leads to heavy-tailed noise independently of input data. For the ResNet model, we observe that Gaussian input leads to Gaussian noise, whereas Wikipedia data leads to be heavy-tailed noise. We thus conclude that the heavy-tailed noise pattern results from both the model architecture as well as the data distribution. 

%% file: sections/discussion.tex
\vspace{-0.2cm}
\section{Discussion}
\vspace{-0.2cm}

One immediate extension from this work is to view RMSProp as a clipping algorithm and prove its convergence under shifting noise. The update for RMSProp and ACClip with $\beta_1=0$ can be written with effective step-sizes $h_{\text{rms}}$ and $h_{\text{clip}}$ respectively as below:
\begin{align*}
x_{k+1}  &= x_{k} - \tfrac{\alpha}{\epsilon + \sqrt{\beta_2 v_k + (1-\beta_2) |g_k|^2}}g_k =:  x_{k} - h_{\text{Adam}} g_k\,, \text{ and}\\
x_{k+1} &= \hspace{3mm} x_{k} - \eta_k  \min\cbr[\big]{\tfrac{\tau_k}{|g_k|}, 1}g_k \hspace{4mm} =:  x_{k} - h_{\text{clip}} g_k.
\end{align*}
Given any set of parameters for RMSProp, if we set the parameters for ACClip as
\[
\eta_k =  \tfrac{2\alpha}{\epsilon +\sqrt{\beta_2 v_k}} \quad \text{and} \quad
\tau_k = \tfrac{\epsilon + \sqrt{\beta_2 v_k}}{\sqrt{1-\beta_2}},
\]
then $ \tfrac{1}{2}h_{\text{clip}} \le  h_{\text{Adam}} \le 2h_{\text{clip}}$. Thus, RMSProp can be seen as ACClip where $\tau_k$ is set using $\textstyle\sqrt{v_k}$, which estimates $\E[|g_k|^2]^{1/2}$, and a correspondingly decreasing step-size. An analysis of RMSprop (and Adam) by viewing them as adaptive clipping methods is a great direction for future work.

In summary, our work theoretically and empirically ties the advantage of adaptive methods over SGD to the heavy-tailed nature of gradient noise. A careful analysis of the noise and its impact yielded two insights: that clipping is an excellent strategy to deal with heavy-tailed noise, and that the ACClip yields state of the art performance for training attention models. Our results add to a growing body of work which demonstrate the importance of the structure of the noise in understanding neural network training. We believe additional such investigations into the source of the heavy tailed-ness, as well as a characterization of the noise can lead to further insights with significant impact on practice.

\section*{Broader impact}
We study convergence rates of gradient methods under a more relaxed noise condition. The result under this setting reaches conclusions that are closer to practice compared to results under the standard setting. Hence, our work provides one way to bridge the theory-practice gap and can facilitate more future works in this direction.

%% file: sections/appendix.tex
\section{Additional definitions and assumptions}\label{sec:app-assump}
Here we describe some of the formal assumptions which were previously skipped. 

\subsection{Assumptions in the nonconvex setting}
We define the standard notion of smoothness.
\begin{assumption}[{\bf$L$-smoothness}]
	\label{assump:smooth}
	$f$ is $L$-smooth, i.e. there exist positive constants  $L$ such that $\forall x, y$, $f(y) \leq f(x) + \inp{\nabla f(x)}{y - x} + \frac{L}{2}\norm{y - x}^2\,.$
\end{assumption}
Note that we only need the smoothness assumption for non-convex functions.

\subsection{Assumptions in the strongly convex setting}
For strongly-convex optimization, instead of bounding the noise, we assume that the stochastic oracle has bounded moment. 
\begin{assumption}[{\bf bounded $\alpha$ moment}]
	\label{assump:alpha-moment-cvx}
	There exists positive real numbers $\alpha \in (1, 2]$ and $G > 0$ such that for all $x$,
	$\E[\|g(x)\|^\alpha] \le G^\alpha$.
\end{assumption}
Note that the above assumption implies a uniform bound on gradient norm. Such bound is necessary for nonsmooth strongly convex problems, as one can no longer factor out the gradient norm using the smoothness assumption. See for example, \cite{rakhlin2011making}. 

\begin{assumption}[{\bf$\mu$-strong-convexity}]
\label{assump:strongly-convex}
 $f$ is $\mu$-strongly convex, if there exist positive constants  $\mu$ such that $\forall x, y$, 
$$ f(y) \ge f(x) + \inner{\nabla f(x), y-x} + \frac{\mu}{2}\|y-x\|^2$$
\end{assumption}

The strong convexity assumption and the bounded gradient assumption implies that the domain is bounded, which we state explicitly below,
\begin{assumption}[{\bf bounded domain}]
	\label{assump:domain}
	We look for a solution $x$ within a bounded convex set $\mathcal{X}$.
\end{assumption}

We didn't upper bound the domain diameter as it is not used explicitly in the proof.  To ensure all updates are within a domain, we use the projected version of ~\eqref{eqn:g-clip} defined as follows:

\begin{align}
x_{k+1} &=  \operatorname{proj}_\mathcal{X} \{x_k - \eta_k \min\cbr[\big]{\tfrac{\tau_k}{\|g_k\|}, 1}g_k\,,\ \tau_k \in \mathbb{R}_{\geq 0} ]\} \label{eqn:proj-g-clip}\tag{proj-GClip}
\end{align}

The projection operator $x = \operatorname{proj}_{\mathcal{X}}(y)$ finds the point $x \in \mathcal{X}$ that has the least distance to $y$.

\section{Effect of global clipping on variance and bias}\label{sec:var-bias}
We focus on \eqref{eqn:g-clip} under stochastic gradients which satisfy Assumption \ref{assump:alpha-moment}.
\begin{lemma}\label{lem:g-clip-bias-variance}
    For any $g(x)$ suppose that assumption \ref{assump:alpha-moment} holds with $\alpha \in (1,2]$. If $\E[\|g(x)\|^\alpha] \le G^\alpha$, then the estimator $\hat{g} := \min\cbr[\big]{\tfrac{\tau_k}{\|g_k\|}, 1}g_k$ from \eqref{eqn:g-clip} with  clipping parameter $\tau \geq 0$ satisfies:
    \[
        \E\sbr*{\norm{\hat{g}(x)}^2} \leq G^\alpha \tau^{2 - \alpha} \text{ and } \norm{\E[\hat{g}(x)] - \nabla f(x)}^2 \leq G^{2\alpha}\tau^{-2(\alpha -1)}\,. 
    \]
\end{lemma}
\begin{proof}
First, we bound the variance.
\begin{align*}
\E[\|\hat{g}(x)\|^2] = \E[\|\hat{g}(x)\|^\alpha \|\hat{g}(x)\|^{2-\alpha}]
\end{align*}
By the fact that $\hat{g}(x) \le \tau$, we get
\begin{align*}
\E[\|\hat{g}(x)\|^2] = \E[\|\hat{g}(x)\|^\alpha \tau^{2-\alpha}] \le G^\alpha \tau^{2-\alpha}.
\end{align*}
Next, we bound the bias,
\begin{align*}
    &\|\E[\hat{g}(x)] - \nabla f(x)\| = \|\E[\hat{g}(x) - g(x)]\| \\
    &\le  \E[\|\hat{g}(x) - g(x)\|] = \E[\|\hat{g}(x) - g(x)\|\indicator{|g(x)| \ge \tau}] \\
    &\le \E\sbr*{\|g(x)\|\indicator{|g(x)| \ge \tau}} \\
    &\le \E\sbr*{\|g(x)\|^\alpha \indicator{|g(x)| \ge \tau}} /\tau^{\alpha-1} .
\end{align*}
The first inequality follows by Jenson's inequality. The second inequality follows by definition of $\hat{g}$. The third inequality follows by $\|g(x)\|^\alpha \indicator{|g(x)| \ge \tau} \ge \|g(x)\| \tau^{\alpha-1} \indicator{|g(x)| \ge \tau} $.
\end{proof}

As we increase the clipping parameter $\tau$, note that the variance (the first term in Lemma \ref{lem:g-clip-bias-variance}) increases while the bias (which is the second term) decreases. This way, we can carefully trade-off the variance of our estimator against its bias, thereby ensuring convergence of the algorithm.

\section{Non-convex Rates (Proof of Theorem~\ref{thm:nonconvex-convex-convergence})}\label{sec:thm-nonconvex}

The lemma in the previous section can be readily used in the nonsmooth strongly convex setting. However, we need a variant of Lemma~\ref{lem:g-clip-bias-variance} in the smooth case. 
\begin{lemma}\label{lem:g-clip-bias-variance-nonconvex}
	For any $g(x)$ suppose that assumption \ref{assump:alpha-moment} holds with $\alpha \in (1,2]$. If $\|\nabla f(x)\| \le \tau/2$, then the estimator $\hat{g} := \min\{1, \tau/\|g_k\|\}g_k$ from \eqref{eqn:g-clip} with global clipping parameter $\tau \geq 0$ satisfies:
	\[
	\E\sbr*{\norm{\hat{g}(x)}^2} \leq 2\|\nabla f(x)\|^2 + 4 \sigma^\alpha \tau^{2 - \alpha} \text{ and } \norm{\E[\hat{g}(x)] - \nabla f(x)}^2 \leq 4\sigma^{2\alpha}\tau^{-2(\alpha -1)}\,. 
	\]
\end{lemma}
\begin{proof}
	First, we bound the variance.
	\begin{align*}
	\E[\|\hat{g}(x)\|^2] &\le  \E[2\|\nabla f(x)\|^2 + 2\|\nabla f(x) - \hat{g}(x)\|^2] \\
	& = \E[2\|\nabla f(x)\|^2 + 2\|\nabla f(x) - \hat{g}(x)\|^\alpha \|\nabla f(x) - \hat{g}(x)\|^{2-\alpha}]  \\
	&\le  \E[2\|\nabla f(x)\|^2 + 2\|\nabla f(x) - \hat{g}(x)\|^\alpha (2\tau)^{2-\alpha}] \\
	&\le 2\|\nabla f(x)\|^2 + 4\tau^{2-\alpha}\E[\|\nabla f(x) - g(x)\|^\alpha]  \\
	& \le 2\|\nabla f(x)\|^2 + 4\tau^{2-\alpha}\sigma^\alpha
	\end{align*}
	The expectation is taken with respect to the randomness in noise. The second last inequality follows by the  fact that $ \|\nabla f(x) - \hat{g}(x)\| < 2\tau$.

	Next, we bound the bias,
	\begin{align*}
	&\|\E[\hat{g}(x)] - \nabla f(x)\| = \|\E[\hat{g}(x) - g(x)]\| \\
	&=  \E[|\|g(x)\| - \tau| \indicator{\|g(x)\| > \tau}]  \\
	&\le \E[\|g(x) - \nabla f(x)\| \indicator{\|g(x)\| > \tau}] \\
	& \le \E[\|g(x) - \nabla f(x)\| \indicator{\|g(x) - \nabla f(x)\| > \tau/2}] \\
	& \le \E[\|g(x) - \nabla f(x)\|^\alpha](\tau/2)^{1 - \alpha } \le 2 \sigma^\alpha \tau^{1 - \alpha }
	\end{align*}
	The last line follows by 
	$$\|g(x) - \nabla f(x)\| \indicator{\|g(x) - \nabla f(x)\| > \tau/2} \le \frac{\|g(x) - \nabla f(x)\|^\alpha }{ (\tau/2)^{\alpha-1}} \indicator{\|g(x) - \nabla f(x)\| > \tau/2}.$$
\end{proof}

Next, we need a subprocedure at the end proof of Lemma 2 from \cite{cutkosky2020momentum}. 
\begin{lemma}[Lemma 2 in \cite{cutkosky2020momentum}]\label{lemma:nonconvex}
	For any vector $v \in \reals^d$, $\inner{v/\|v\|, \nabla f(x)} \ge \frac{\|\nabla f(x)\|}{3} - \frac{8\|v - \nabla f(x)\|}{3}$.
\end{lemma}

Finally, we are ready to show the proof.

\begin{proof}
At each iteration, we consider two cases, either $\|\nabla f(x_k)\| < \tau /2$ or $\|\nabla f(x_k)\| \ge \tau /2$.

\paragraph{Case 1: $\|\nabla f(x_k)\| < \tau /2$}

For simplicity, we denote $\hat{g}_k = \min\{1, \tau/\|g_k\|\}g_k$ and the bias $ b_k = \E[\hat{g}_k] - \nabla f(x_k) $. By Assumption~\ref{assump:smooth}, we have 
\begin{align*}
f(x_{k}) &\leq f(x_{k-1}) + \inp{\nabla f(x_{k-1})}{-\eta_k \hat{g}_{k-1}} + \frac{\eta_{k-1}^2 L}{2}\norm{\hat{g}_{k-1}}^2 \\
&\leq f(x_{k-1}) -\eta_{k-1} \|\nabla f(x_{k-1})\|^2 - \eta_{k-1} \inp{\nabla f(x_{k-1})}{b_{k-1}} + \frac{\eta_{k-1}^2 L}{2}\norm{\hat{g}_{k-1}}^2 \\
&\leq f(x_{k-1}) -\eta_{k-1} \|\nabla f(x_{k-1})\|^2 - \eta_{k-1} \inp{\nabla f(x_{k-1})}{b_{k-1}} + \frac{\eta_{k-1}^2 L}{2}\norm{\hat{g}_{k-1}}^2\\
&\leq f(x_{k-1}) -\eta_{k-1} \|\nabla f(x_{k-1})\|^2 + \frac{\eta_{k-1}}{2} \norm*{\nabla f(x_{k-1})}^2 + \frac{\eta_{k-1}}{2}\norm{b_{k-1}}^2 + \frac{\eta_{k-1}^2 L}{2}\norm{\hat{g}_{k-1}}^2\,.
\end{align*}
Here the last step used the AM-GM inequality. Then, taking expectation in both sides and using Lemma \ref{lem:g-clip-bias-variance-nonconvex} gives
\begin{align*}
    \E\sbr*{f(x_{k}) | x_{k-1}} &\leq f(x_{k-1}) - (\frac{\eta_{k-1}}{2} - \eta L) \norm*{\nabla f(x_{k-1})}^2 + 2\eta_{k-1}^2 L \sigma^\alpha \tau^{2-\alpha} + \frac{2\eta_k \sigma^{2\alpha}}{\tau^{2\alpha -2}}\,\\
    &\leq f(x_{k-1}) - \frac{\eta_{k-1}}{4} \norm*{\nabla f(x_{k-1})}^2 + 2\eta_{k-1}^2 L \sigma^\alpha \tau^{2-\alpha} + \frac{2\eta_{k-1} \sigma^{2\alpha}}{\tau^{2\alpha -2}}\,.
\end{align*}
In the last step we used $\{\eta_k = \eta \le  \frac{1}{4L}\}$.

\paragraph{Case 2: $\|\nabla f(x_k)\| > \tau /2$} Recall $\hat{g}_k = \min\{1, \tau/\|g_k\|\}g_k$ and parameter choices $\eta_k = \eta = \min\{\frac{1}{4L}, \frac{1}{L\tau^\alpha} , \frac{1}{24L\tau}\}$ and $\tau_k = \tau = \max\{2, 48^{1/(\alpha-1)}\sigma^{\alpha/(\alpha - 1)}, 8\sigma, \sigma K^{\frac{1}{3\alpha - 2}}\}$. We use $\nabla f$ as a shorthand for $\nabla f(x_{k})$:

\begin{align}\label{eq:case2}
&\E[\inp{\nabla f}{g_k} \mathbbm{1}_{\{\|g_k\| \le \tau \}}]  \ge \E[(\|\nabla f\|^2 - \|\nabla f\|\|g_k - \nabla f\| )\mathbbm{1}_{\{\|g_k\| \le \tau \}}] \nonumber \\
\ge & \E[\|\nabla f\|^2\mathbbm{1}_{\{\|g_k\| \le \tau \}} - \tfrac{1}{2}\|\nabla f\|^2\mathbbm{1}_{\{\|g_k\| \le \tau, \|g_k - \nabla f\| \le \tau/4 \}} - \|\nabla f\|\|g_k - \nabla f\|\mathbbm{1}_{\{\|g_k\| \le \tau, \|g_k - \nabla f\| \ge \tau/4 \}}] \nonumber \\
\ge & \tfrac{p}{2} \|\nabla f\|^2 - \|\nabla f\| \E[\|g_k - \nabla f\|\mathbbm{1}_{\{ \|g_k - \nabla f\| \ge \tau/4 \}}] 
\ge \tfrac{p}{2} \|\nabla f\|^2 - \|\nabla f\|
\tfrac{\sigma^\alpha}{(\tau/4)^{\alpha - 1}}
\end{align}
The first inequality uses 
$$\inp{\nabla f}{g_k} = \|\nabla f\|^2 + \inp{\nabla f}{g_k - \nabla f}.$$ 
The second line follows by 
$$\|\nabla f\| > \tau/2 \text{ and } \|g_k - \nabla f\|< \tau/4 \implies -\|\nabla f\|\|g_k - \nabla f\| \ge -\|\nabla f\|^2/2.$$
The last inequality follows by $\sigma^\alpha \ge \E[\|g_k - \nabla f\|^\alpha] \ge \E[\|g_k - \nabla f\|(\frac{\tau}{4})^{\alpha - 1} \mathbbm{1}_{ \|g_k - \nabla f\| \ge \tau/4 \}}]$ . 

With the above, we get
\begin{align*}
\E[\inp{\nabla f}{\hat{g}_{k}}] &= \E[\inp{\nabla f}{g_k} \mathbbm{1}\{\|g_k\| \le \tau \}] + \E[\inp{\nabla f}{g_k/\|g_k\|} \mathbbm{1}\{\|g_k\| \ge \tau \}]\\
&\ge \tfrac{p}{2} \|\nabla f\|^2 - \|\nabla f\| \tfrac{\sigma^\alpha}{(\tau/4)^{\alpha - 1}} + (1-p) \|\nabla f\|/3 - \tfrac{8}{3}\E[\|\nabla f - g_k\|] \\
& \ge \|\nabla f\|/3 - \|\nabla f\| \tfrac{\sigma^\alpha}{(\tau/4)^{\alpha - 1}}  - \tfrac{8\sigma}{3} \\
&\ge \|\nabla f\|/3 - \|\nabla f\|/12 - \|\nabla f\|/6 \\
& \ge \|\nabla f\|/12
\end{align*}
The second line follows by Lemma~\ref{lemma:nonconvex} and ~\eqref{eq:case2}. The third line follows by  that  $\tau \ge 2$, and $\|\nabla f\| \ge \tau/2$ imply $\tfrac{p}{2} \|\nabla f\|^2 \ge p\|\nabla f\|/3.$ Then, by $\tau \ge 48^{1/(\alpha-1)}\sigma^{\alpha/(\alpha - 1)}$, we have $\tfrac{\sigma^\alpha}{(\tau/4)^{\alpha - 1}} \le \tfrac{1}{12}$. By $\tau \ge 8\sigma$, $\tfrac{8}{3}\sigma \le \tau/3 \le \|\nabla f\|/6$.

\begin{align*}
\E[f(x_{k})] &\leq f(x_{k-1}) + \E[\inp{\nabla f(x_{k-1})}{-\eta_k \hat{g}_{k}} ] + \frac{\eta_k^2 L}{2}\tau^2 \\
&\leq f(x_{k-1}) -\eta_k \|\nabla f(x_{k-1})\|/12 +  \eta_k^2 L\tau \|\nabla f(x_{k-1})\| \\
& \le f(x_{k-1}) -\eta_k \|\nabla f(x_{k-1})\|/24
\end{align*}
The last inequality above follows by $\frac{1}{24L\tau}$.

Combine the two cases we have
\begin{align*}
\E\sbr*{f(x_{k}) | x_{k-1}} &\leq f(x_{k-1}) - \frac{\eta}{24} \min\{\norm*{\nabla f(x_{k-1})}^2, \norm*{\nabla f(x_{k-1})}\} + 2\eta^2 L \sigma^\alpha \tau^{2-\alpha} + \frac{2\eta \sigma^{2\alpha}}{\tau^{2\alpha -2}}\,.
\end{align*}

Rearrange and sum the terms above for some fixed step-size and threshold $\{\tau_k = \tau\}$ to get
\begin{align*}
    \frac{1}{K}\sum_{k=1}^K \E\sbr*{\min\{\norm{\nabla f(x_{k-1})}^2, \norm{\nabla f(x_{k-1})}\}} &\leq \frac{24}{\eta K}(f(x_0) - \E[f(x_{K})]) + 48\eta L \sigma^\alpha \tau^{2-\alpha}+ 48\frac{\sigma^{2\alpha}}{\tau^{2\alpha -2}}\\
    &\leq \underbrace{\frac{24}{\eta K}(f(x_0) - f^\star)}_{T_1} + \underbrace{48\eta L \sigma^\alpha \tau^{2-\alpha}+ \frac{48\sigma^{2\alpha}}{\tau^{2\alpha -2}}}_{T_2}\,.
\end{align*}
Since we use a stepsize $\eta \le \frac{1}{L\tau^\alpha}$, we can simplify $T_2$ as
\[
    \eta L \sigma^\alpha \tau^{2-\alpha}+ \frac{\sigma^{2\alpha}}{\tau^{2\alpha -2}} \le \frac{ \sigma^{2\alpha} + \sigma^{\alpha}}{\tau^{2\alpha - 2}} \,.
\]

Denote $F_0 = f(x_0) - f^\star$ to ease notation. Then, adding $T_2$ back to $T_1$ and using a threshold $\tau \ge  \sigma K^{\frac{1}{3\alpha - 2}}$ we get
\begin{align*}
    T_1 + T_2 &\le \frac{24 F_0 }{ K} (L \tau^\alpha + 4L + 24L\tau)+ 48 \frac{ \sigma^{2\alpha} + \sigma^{\alpha}}{\tau^{2\alpha - 2}} \\
    & \le 48 (\sigma^2 + \sigma^{2-\alpha})K^{\frac{-2\alpha+2}{3\alpha - 1}}  \\
    &+ 24F_0 L K^{-1} (4+  \max\{4, 48^{\alpha/(\alpha-1)}\sigma^{\alpha^2 / (\alpha - 1)}, 64\sigma^\alpha, \sigma^\alpha K^{\frac{\alpha}{3\alpha - 2}}\}) \\
    & + 24F_0 L K^{-1} (\max\{2, 48^{1/(\alpha-1)}\sigma^{\alpha / (\alpha - 1)}, 8\sigma, \sigma K^{\frac{1}{3\alpha - 2}}\})\\
    & = \mathcal{O}(K^{\frac{-2\alpha+2}{3\alpha - 1}})\end{align*}
This proves the statement of the theorem. 
\end{proof}

\section{Strongly-Convex Rates (Proof of Theorem~\ref{thm:strongly-convex-convergence})}\label{sec:thm-strongly}
For simplicity, we denote $\hat{g}_k = \min\cbr[\big]{\tfrac{\tau_k}{\|g_k\|}, 1}g_k$ and the bias $ b_k = \E[\hat{g}_k] - \nabla f(x_k) $.
\begin{align*}
\|x_k - x^*\|^2 &= \|\operatorname{proj}_{\mathcal{X}} (x_{k-1} - \eta_k\hat{g}_{k-1} - x^* )\|^2 \\
&\le \|(x_{k-1} - \eta_k\hat{g}_{k-1} - x^* )\|^2 \\
&= \|x_{k-1} - x^*\|^2 - 2\eta_k \inp*{x_{k-1}-x^*}{\nabla f(x_{k-1})}\\
&  - 2\eta_k \inp*{x_{k-1}-x^*}{b_{k-1}} + \eta_k^2 \|\hat{g}_{k-1}\|^2\\
& \le (1 - \mu \eta_k)\|x_{k-1} - x^*\|^2 - 2\eta_k(f(x_{k-1}) - f^*))\\
&  + 2\eta_k (\frac{\mu}{4}\norm{x_{k-1}-x^*}^2 + \frac{4}{\mu} \norm*{b_k}^2) + \eta_k^2 \|\hat{g}_{k-1}\|^2.
\end{align*}
The first inequality follows by the nonexpansivity of projections onto convex sets.

Rearrange and we get 
\begin{align*}
f(x_{k-1}) - f^* \le \frac{\eta_k^{-1} - \mu/2}{2}\|x_{k-1} - x^*\|^2 -   \frac{\eta_k^{-1}}{2}\|x_{k} - x^*\|^2 + \frac{4}{\mu}\norm{b_k}^2 + \frac{\eta_k}{2} \|\hat{g}_{k-1}\|^2.
\end{align*}

After taking expectation and apply the inequality from Lemma~\ref{lem:g-clip-bias-variance}, we get
\begin{align*}
\E\sbr*{f(x_{k-1})} - f^* &\le \E\sbr*{\frac{\eta_k^{-1} - \mu/2}{2}\|x_{k-1} - x^*\|^2 -   \frac{\eta_k^{-1}}{2}\|x_{k} - x^*\|^2}\\
&+ 4G^{2\alpha}\tau^{2-2\alpha}\mu^{-1} + \eta_k G^{\alpha}\tau^{2-\alpha}/2.
\end{align*}

Then take $\eta_k = \frac{4}{\mu (k+1)}, \tau_k = Gk^{\frac{1}{\alpha}}$ and multiply both side by $k$, we get
\begin{align*}
k\E\sbr*{f(x_{k-1})} - f^* &\le \frac{\mu}{8}\E\sbr*{k(k-1)\|x_{k-1} - x^*\|^2 -   k(k+1)\|x_{k} - x^*\|^2}\\
&+ 8G^{2}k^{\frac{2-\alpha}{\alpha}}\mu^{-1}.
\end{align*}
Notice that $\sum_{k=1}^K k^{\frac{2-\alpha}{\alpha}}\le \int_0^{K+1}k^{\frac{2-\alpha}{\alpha}}dk  \le (K + 1)^{2/\alpha}$. Sum over $k$ and we get
\begin{align*}
\tsum_{k=1}^K k\E\sbr*{f(x_{k-1})} - f^* &\le \frac{\mu}{8}\E\sbr*{ - T(T+1)\|x_{T} - x^*\|^2}\\
&+ 8G^{2}(K+1)^{\frac{2}{\alpha}}\mu^{-1}.
\end{align*}
Devide both side by $\frac{K(K+1)}{2}$ and we get

\begin{align*}
\frac{2}{K(K+1)}\tsum_{k=1}^K k\E\sbr*{f(x_{k-1})} - f^* &\le 8G^{2}K^{-1}(K+1)^{\frac{2-\alpha}{\alpha}}\mu^{-1}.
\end{align*}

Notice that for $K \ge 1$, $K^{-1} \le 2 (K+1)^{-1}$. We have
\begin{align*}
\frac{2}{K(K+1)}\tsum_{k=1}^K k\E\sbr*{f(x_{k-1})} - f^* &\le 16G^{2}(K+1)^{\frac{2-2\alpha}{\alpha}}\mu^{-1}.
\end{align*}
The theorem then follows by Jensen's inequality. \qed

\section{Effect of coordinate-wise moment bound}\label{sec:coord-moment}
We now examine how the rates would change if we replace Assumption~\ref{assump:alpha-moment-cvx} with Assumption~\ref{assump:coord-noise}.
\subsection{Convergence of GClip (proof of Corollary~\ref{cor:coord-g-clip-convergence})}\label{sec:cor-gclip}
We now look at\eqref{eqn:g-clip} under assumption \ref{assump:coord-noise}.


The proof of both the convex and non-convex rates following directly from the following Lemma.
\begin{lemma}\label{lemma:g-clip}
    For any $g(x)$ suppose that assumption \ref{assump:coord-noise} with $\alpha \in (1,2]$. Then suppose we have a constant upper-bound
    \[
        \E[\norm*{g(x)}^\alpha] \le D\,.
    \]
    Then $D$ satisfies
    \[
        d^{\frac{\alpha}{2}-1}\norm{B}_\alpha^\alpha \leq D \leq d^{\alpha/2} \norm{B}_\alpha^\alpha.
    \]
\end{lemma}
\begin{proof}
Note that the function $(\cdot)^{\alpha/2}$ is concave for $\alpha \in (1,2]$. Using Jensen's inequality we can rewrite as:
\[
D \geq \E[\norm*{g(x)}^\alpha] = d^{\alpha/2}\E\sbr*{\rbr*{\frac{1}{d}\sum_{i=1}^d |g(x)^{(i)}|^2 }^{\alpha/2}}  \geq d^{\alpha/2 -1 }\E\sbr*{\sum_{i=1}^d |g(x)^{(i)}|^\alpha} \,.
\]
Since the right hand-side can be as large as $d^{\frac{\alpha}{2}-1}\norm{B}_\alpha^\alpha$, we have our first inequality.
On the other hand, we also have an upper bound below:

\begin{align*}
    \E[\norm*{g(x)}^\alpha] = \E\sbr*{\rbr*{\sum_{i=1}^d |g(x)^{(i)}|^2}^{\alpha/2}} \le \E\sbr*{\rbr*{d\rbr{\max_{i=1}^d g(x)^{(i)}}^2}^{\alpha/2}}
    \\\le \E\sbr*{d^{\alpha/2} \rbr{\max_{i=1}^d g(x)^{(i)}}^\alpha} 
    \le \E\sbr*{d^{\alpha/2} \sum_{i=1}^d \rbr{g(x)^{(i)}}^\alpha} \le d^{\alpha/2} \sum_{i=1}^d B_i^\alpha
\end{align*}
where $\norm{B}^\alpha_\alpha = \sum_{i=1}^d B_i^\alpha$. Thus, we have shown that
\[
    d^{\frac{\alpha}{2} - 1} \norm{B}^\alpha_\alpha \leq \E[\norm*{g(x)}^\alpha] \leq  d^{\alpha/2}\norm{B}^\alpha_\alpha \,.
\]
We know that Jensen's inequality is tight when all the co-ordinates have equal values. This means that if the noise across the coordinates is linearly correlated the lower bound is tighter, whereas the upper bound is tighter if the coordinates depend upon each other in a more complicated manner or are independent of each other.
\end{proof}
Substituting this bound on $G$ in Theorems~\ref{thm:strongly-convex-convergence} and~\ref{thm:nonconvex-convex-convergence} gives us our corollaries.

\subsection{Convergence of CClip (Proof of Theorem~\ref{thm:c-clip-convergence})}

The proof relies on the key lemma which captures the bias-variance trade off under the new noise-assumption and coordinate-wise clipping.
\begin{lemma}\label{lem:c-clip-bias-variance}
    For any $g(x)$ suppose that assumption \ref{assump:coord-noise} with $\alpha \in (1,2]$ holds. Denote $g_i$ to be $i_{th}$ component of $g(x)$, $\nabla f(x)_i$ to be $i_{th}$ component of $\nabla f(x)$. Then the estimator $\hat{g}(x) = [\hat{g}_1; \cdots; \hat{g}_d]$ from \eqref{eqn:c-clip} with clipping parameter $\tau = [\tau_1; \tau_2; \cdots ; \tau_d]$ satisfies:
    
    \[
        \E\sbr*{\norm{\hat{g}_i}^2} \leq B_i^\alpha \tau_i^{2 - \alpha} \text{ and } \norm{\E[\hat{g}_i] - \nabla f(x)_i}^2 \leq B_i^{2\alpha}\tau_i^{-2(\alpha -1)}\,. 
    \]
\end{lemma}
\begin{proof}
Apply Lemma~\ref{lem:g-clip-bias-variance} to the one dimensional case in each coordinate.
\end{proof}

\begin{proof}[Proof of Theorem~\ref{thm:c-clip-convergence}]
Theorem~\ref{thm:strongly-convex-convergence}
For simplicity, we denote $\hat{g}_k = \eta_k\hat{g}(x_k)$ and the bias $ b_k = \E[\hat{g}_k] - \nabla f(x_k) $.
\begin{align*}
    \|x_k - x^*\|^2 &= \|x_{k-1} - \eta_k\hat{g}_{k-1} - x^*\|^2 \\
    &= \|x_{k-1} - x^*\|^2 - 2\eta_k \inp*{x_{k-1}-x^*}{\nabla f(x_{k-1})}\\
    &  - 2\eta_k \inp*{x_{k-1}-x^*}{b_{k-1}} + \eta_k^2 \|\hat{g}_{k-1}\|^2\\
    & \le (1 - \mu \eta_k)\|x_{k-1} - x^*\|^2 - 2\eta_k(f(x_{k-1}) - f^*))\\
    &  + 2\eta_k (\frac{\mu}{4}\norm{x_{k-1}-x^*}^2 + \frac{4}{\mu} \norm*{b_k}^2) + \eta_k^2 \|\hat{g}_{k-1}\|^2.
\end{align*}

Rearrange and we get 
\begin{align*}
    f(x_{k-1}) - f^* \le \frac{\eta_k^{-1} - \mu/2}{2}\|x_{k-1} - x^*\|^2 -   \frac{\eta_k^{-1}}{2}\|x_{k} - x^*\|^2 + \frac{4}{\mu}\norm{b_k}^2 + \frac{\eta_k}{2} \|\hat{g}_{k-1}\|^2.
\end{align*}

After taking expectation and apply the inequality from Lemma~\ref{lem:g-clip-bias-variance}, we get
\begin{align*}
    \E\sbr*{f(x_{k-1})} - f^* &\le \E\sbr*{\frac{\eta_k^{-1} - \mu/2}{2}\|x_{k-1} - x^*\|^2 -   \frac{\eta_k^{-1}}{2}\|x_{k} - x^*\|^2}\\
    &+ \tsum_{i=1}^d 4B_i^{2\alpha}\tau_i^{2-2\alpha}\mu^{-1} + \eta_k G^{\alpha}\tau_i^{2-\alpha}/2.
\end{align*}

Then take $\eta_k = \frac{4}{\mu (k+1)}, \tau_i = B_i k^{\frac{1}{\alpha}}$ and multiply both side by $k$, we get
\begin{align*}
    k\E\sbr*{f(x_{k-1})} - f^* &\le \frac{\mu}{8}\E\sbr*{k(k-1)\|x_{k-1} - x^*\|^2 -   k(k+1)\|x_{k} - x^*\|^2}\\
    &+ 8\tsum_{i=1}^d B_i^{2}k^{\frac{2-\alpha}{\alpha}}\mu^{-1}.
\end{align*}
Notice that $\sum_{k=1}^K k^{\frac{2-\alpha}{\alpha}}\le \int_0^{K+1}k^{\frac{2-\alpha}{\alpha}}dk  \le (K + 1)^{2/\alpha}$. Sum over $k$ and we get
\begin{align*}
    \tsum_{k=1}^K k\E\sbr*{f(x_{k-1})} - f^* &\le \frac{\mu}{8}\E\sbr*{ - T(T+1)\|x_{T} - x^*\|^2}\\
    &+ 8\tsum_{i=1}^d B_i^{2}k^{\frac{2-\alpha}{\alpha}}\mu^{-1}.
\end{align*}
Devide both side by $\frac{K(K+1)}{2}$ and we get

\begin{align*}
    \frac{2}{K(K+1)}\tsum_{k=1}^K k\E\sbr*{f(x_{k-1})} - f^* &\le 8\tsum_{i=1}^d B_i^{2}k^{\frac{2-\alpha}{\alpha}}\mu^{-1}.
\end{align*}

Notice that for $K \ge 1$, $K^{-1} \le 2 (K+1)^{-1}$. We have
\begin{align*}
    \frac{2}{K(K+1)}\tsum_{k=1}^K k\E\sbr*{f(x_{k-1})} - f^* &\le 16\tsum_{i=1}^d B_i^{2}k^{\frac{2-\alpha}{\alpha}}\mu^{-1}.
\end{align*}
The theorem then follows by Jensen's inequality. \qed
\end{proof}

%

\section{Lower Bound (Proof of Theorem~\ref{thm:lower-bound})}\label{sec:lower-bound}
We consider the following simple one-dimensional function class parameterized by $b$:
\begin{equation}\label{eqn:function-class}
    \min_{x \in [0,1/2]} \cbr*{f_{b}(x) = \tfrac{1}{2}(x - b)^2}\,, \text{ for } b \in [0,1/2]\,.
\end{equation}
Also suppose that for $\alpha \in (1,2]$ and $b \in [0,1/2]$ the stochastic gradients are of the form:
\begin{equation}\label{eqn:noise-class}
    g(x) \sim \nabla f_b(x) + \chi_{b}\,, \E[g(x)] = \nabla f_b(x)\,, \text{ and } \E[|g(x)|^\alpha] \leq 1\,.
\end{equation}
Note that the function class \eqref{eqn:function-class} has $\mu=1$ and optimum value $f_b(b) = 0$, and the $\alpha$-moment of the noise in \eqref{eqn:noise-class} satisfies $G = B \leq 1$. 
Thus, we want to prove the following:

\begin{theorem}
For any $\alpha \in (1,2]$ there exists a distribution $\chi_{b}$ such that the stochastic gradients satisfy \eqref{eqn:noise-class}. Further, for any (possibly randomized) algorithm $\Ac$, define $\Ac_k\rbr*{f_b + \chi_{b}}$ to be the output of the algorithm $\Ac$ after $k$ queries to the stochastic gradient $g(x)$. Then:
\[
    \max_{b \in [0,1/2]} \E[f_b(\Ac_k(f_b + \chi_{b}))] \geq \Omega\rbr*{\frac{1}{k^{2(\alpha -1)/\alpha}}}\,.
\]
\end{theorem}

Our lower bound construction is inspired by Theorem~2 of \cite{bubeck2013bandits}. Let $\Ac_k(f_b + \chi_{b})$ denote the output of any possibly randomized algorithm $\Ac$ after processing $k$ stochastic gradients of the function $f_b$ (with noise drawn i.i.d. from distribution $\chi_{b}$). Similarly, let $\Dc_k(f_b + \chi_{b})$ denote the output of a \emph{deterministic} algorithm after processing the $k$ stochastic gradients. Then from Yao's minimax principle we know that for any fixed distribution $\Bc$ over $[0,1/2]$,
\[
    \min_{\Ac}\max_{b \in [0,1/2]}\E_{\Ac}\sbr*{ \E_{\chi_b}f_b(\Ac_k(f_b + \chi_{b}))} \geq \min_{\Dc} \E_{b \sim \Bc}\sbr*{\E_{\chi_b}f_b(\Dc_k(f_b + \chi_{b}))}\,.
\]
Here we denote $\E_{\Ac}$ to be expectation over the randomness of the algorithm $\Ac$ and $\E_{\chi_b}$ to be over the stochasticity of the the noise distribution $\chi_b$.
Hence, we only have to analyze deterministic algorithms to establish the lower-bound. Further, since $\Dc_k$ is deterministic, for any \emph{bijective} transformation $h$ which transforms the stochastic gradients, there exists a deterministic algorithm $\tDc$ such that $\tDc_k(h(f_b + \chi_{b})) = \Dc_k(f_b + \chi_{b})$. This implies that for any bijective transformation $h(\cdot)$ of the gradients:
\[
    \min_{\Dc} \E_{b \sim \Bc}\sbr*{\E_{\chi_b}f_b(\Dc_k(f_b + \chi_{b}))} =   \min_{\Dc} \E_{b \sim \Bc}\sbr*{\E_{\chi_b} f_b(\Dc_k(h(f_b + \chi_{b})))}\,.
\]
In this rest of the proof, we will try obtain a lower bound for the right hand side above.

We now describe our construction of the three quantities to be defined: the problem distribution $\Bc$, the noise distribution $\chi_{b}$, and the bijective mapping $h(\cdot)$. All of our definitions are parameterized by $\alpha \in (1,2]$ (which is given as input) and by $\epsilon \in (0,1/8]$ (which represents the desired target accuracy). We will pick $\epsilon$ to be a fixed constant which depends on the problem parameters (e.g. $k$) and should be thought of as being small.
\begin{itemize}
    \item{Problem distribution:} $\Bc$ picks $b_0 = 2\epsilon$ or $b_1 = \epsilon$ at random i.e. $\nu \in \{0,1\}$ is chosen by an unbiased coin toss and then we pick
    \begin{equation}\label{eqn:lower-prob}
        b_{\nu} = (2-\nu)\epsilon \,.
    \end{equation}
    \item{Noise distribution:} Define a constant $\gamma = (4\epsilon)^{1/(\alpha - 1)}$ and $p_\nu = (\gamma^\alpha - 2\nu\gamma\epsilon)$. Simple computations verify that $\gamma \in (0,1/2]$ and that
    \[
        p_\nu = (4\epsilon)^{\frac{\alpha}{\alpha - 1}} - 2\nu\rbr*{4 \epsilon^\alpha}^{\frac{1}{\alpha - 1}}  = (4-2\nu)\rbr*{4 \epsilon^\alpha}^{\frac{1}{\alpha - 1}} \in (0,1)\,.
    \]
    Then, for a given $\nu \in \{0,1\}$ the stochastic gradient $g(x)$ is defined as
    \begin{equation}\label{eqn:lower-noise}
        g(x) = \begin{cases}
                x - \frac{1}{2\gamma} &\text{ with prob. } p_\nu\,,\\
                x &\text{ with prob. } 1 - p_\nu\,.
            \end{cases}
    \end{equation}
    To see that we have the correct gradient in expectation verify that 
    \[
        \E[g(x)] = x - \frac{p_\nu}{2\gamma} = x - \frac{\gamma^{\alpha - 1}}{2} +\nu \epsilon = x - (2-\nu)\epsilon = x - b_\nu = \nabla f_{b_{\nu}}(x)\,.
    \]
    Next to bound the $\alpha$ moment of $g(x)$ we see that
    \[
        \E[\abs{g(x)}^\alpha] \leq \gamma^\alpha\rbr*{x - \frac{1}{2\gamma}}^{\alpha} + x^{\alpha} \leq \frac{1}{2} + \frac{1}{2} = 1\,.
    \]
    The above inequality used the bounds that $\alpha \geq 1$, $x \in [0,1/2]$, and $\gamma \in (0, 1/2]$. Thus $g(x)$ defined in \eqref{eqn:lower-noise} satisfies condition \eqref{eqn:noise-class}.
    \item{Bijective mapping:} Note that here the only unknown variable is $\nu$ which only affects $p_\nu$. Thus the mapping is bijective as long as the \emph{frequencies} of the events are preserved. Hence given a stochastic gradient $g(x_i)$ the mapping we use is:
    \begin{equation}\label{eqn:lower-map}
        h(g(x_i)) = \begin{cases}
                        1 &\text{ if } g(x_i) = x_i - \frac{1}{2\gamma}\,,\\
                        0 &\text{ otherwise.}
                        \end{cases}
    \end{equation}
\end{itemize}
Given the definitions above, the output of algorithm $\Dc_k$ is thus simply a function of $k$ i.i.d. samples drawn from the Bernoulli distribution with parameter $p_\nu$ (which is denoted by $\ber(p_\nu)$). We now show how achieving a small optimization error implies being able to guess the value of $\nu$.

\begin{lemma}\label{lem:guessing-game}
Suppose we are given problem and noise distributions defined as in \eqref{eqn:lower-prob} and \eqref{eqn:lower-noise}, and an bijective mapping $h(\cdot)$ as in \eqref{eqn:lower-map}. Further suppose that there is a deterministic algorithm $\Dc_k$ whose output after processing $k$ stochastic gradients satisfies
\[
    \E_{b \sim \Bc}\sbr*{\E_{\chi_b} f_b(\Dc_k(h(f_b + \chi_{b})))} < \epsilon^2/64\,.
\]
Then, there exists a deterministic function $\tDc_k$ which given $k$ independent samples of $\ber(p_\nu)$ outputs $\nu' = \tDc_k(\ber(p_\nu)) \in \{0,1\}$ such that
\[
    \Pr\sbr*{\tDc_k(\ber(p_\nu)) = \nu} \geq \frac{3}{4}\,.
\]
\end{lemma}
\begin{proof}
Suppose that we are given access to $k$ samples of $\ber(p_\nu)$. Use these $k$ samples as the input $h(f_b + \chi_{b}))$ to the procedure $\Dc_k$ (this is valid as previously discussed), and let the output of $\Dc_k$ be $x^{(\nu)}_k$. The assumption in the lemma states that
\[
    \E_\nu\sbr*{\E_{\chi_b} \abs{x^{(\nu)}_k - b_\nu}^2} < \frac{\epsilon^2}{32} \text{, which implies that } \E_{\chi_b}\abs{x^{(\nu)}_k - b_\nu}^2 < \frac{\epsilon^2}{16} \text{ almost surely.}
\]
Then, using Markov's inequality (and then taking square-roots on both sides) gives
\[
    \Pr\sbr*{\abs{x^{(\nu)}_k - b_\nu} \geq \frac{\epsilon}{2}} \leq \frac{1}{4}\,.
\]
Consider a simple procedure $\tDc_k$ which outputs $\nu' = 0$ if $x^{(\nu)}_k \geq \frac{3\epsilon}{2}$, and $\nu' = 1$ otherwise. Recall that $\abs{b_0 - b_1} = \epsilon$ with $b_0 = 2\epsilon$ and $b_1 = \epsilon$. With probability $\frac{3}{4}$, $\abs{x^{(\nu)}_k - b_\nu} < \frac{\epsilon}{2}$ and hence the output $\nu'$ is correct.
\end{proof}

Lemma~\ref{lem:guessing-game} shows that if the optimization error of $\Dc_k$ is small, there exists a procedure $\tDc_k$ which distinguishes between the Bernoulli distributions with parameters $p_0$ and $p_1$ using $k$ samples. To argue that the optimization error is large, one simply has to argue that a large number of samples are required to distinguish between $\ber(p_0)$ and $\ber(p_1)$.
\begin{lemma}\label{lem:sample-complex}
    For any deterministic procedure $\tDc_k(\ber(p_\nu))$ which processes $k$ samples of $\ber(p_\nu)$ and outputs $\nu'$
    \[
        \Pr\sbr*{\nu' = \nu} \leq \frac{1}{2} + \sqrt{k \rbr*{4\epsilon}^{\frac{\alpha}{\alpha - 1}}}\,.
    \]
\end{lemma}
\begin{proof}
Here it would be convenient to make the dependence on the samples explicit. Denote $\ss^{(\nu)}_k = \rbr*{s_1^{(\nu)},\dots,s_k^{(\nu)}} \in \{0,1\}^k$ to be the $k$ samples drawn from $\ber(p_\nu)$ and denote the output as $\nu' = \tDc(\ss^{(\nu)}_k)$. With some slight abuse of notation where we use the same symbols to denote the realization and their distributions, we have:
\[
    \Pr\sbr*{\tDc(\ss^{(\nu)}_k) = \nu} = \frac{1}{2}\Pr\sbr*{\tDc(\ss^{(1)}_k) = 1} + \frac{1}{2}\Pr\sbr*{\tDc(\ss^{(0)}_k) = 0} = \frac{1}{2} + \frac{1}{2}\E\sbr*{\tDc(\ss^{(1)}_k) - \tDc(\ss^{(0)}_k)}\,.
\]
Next using Pinsker's inequality we can upper bound the right hand side as:
\[\E\sbr*{\tDc(\ss^{(1)}_k) - \tDc(\ss^{(0)}_k)} \leq \abs*{\tDc(\ss^{(1)}_k) - \tDc(\ss^{(0)}_k)}_{TV} \leq \sqrt{\frac{1}{2}\KL\rbr*{\tDc\rbr*{\ss^{(1)}_k}, \tDc\rbr*{\ss^{(0)}_k}}}\,,    
\]
where $\abs{\cdot}_{TV}$ denotes the total-variation distance and $\KL(\cdot, \cdot)$ denotes the KL-divergence. Recall two properties of KL-divergence: i) for a product measures defined over the same measurable space $(p_1,\dots,p_k)$ and $(q_1,\dots,q_k)$, 
\[
    \KL((p_1,\dots,p_k), (q_1,\dots,q_k)) = \sum_{i=1}^k \KL(p_i,q_i)\,,
\] and ii) for any deterministic function $\tDc$, 
\[
    \KL(p,q) \geq \KL(\tDc(p), \tDc(q))\,.
\]
Thus, we can simplify as
\begin{align*}
    \Pr\sbr*{\tDc(\ss^{(\nu)}_k) = \nu} &\leq  \frac{1}{2} +  \sqrt{\frac{k}{8}\KL\rbr*{\ber(p_1), \ber(p_0)}}\\
    &\leq  \frac{1}{2} + \sqrt{\frac{k}{8}\frac{(p_0 - p_1)^2}{p_0 ( 1- p_0)}}\\
    &\leq  \frac{1}{2} + \sqrt{\frac{k(\gamma \epsilon)^2}{4\gamma^\alpha}}\\
    &=  \frac{1}{2} +  \sqrt{k \rbr*{4^{(2 - 1/\alpha)}\epsilon}^{\frac{\alpha}{\alpha - 1}}}\,.
\end{align*}
Recalling that $\alpha \in (1,2]$ gives us the statement of the lemma.
\end{proof}
If we pick $\epsilon$ to be
\[
    \epsilon = \frac{1}{16 k^{(\alpha - 1)/\alpha}}\,,
\]
we have that
\[
    \frac{1}{2} + \sqrt{ k \rbr*{4\epsilon}^{\frac{\alpha}{\alpha - 1}}} < \frac{3}{4}\,.
\]
Given Lemmas~\ref{lem:guessing-game} and~\ref{lem:sample-complex}, this implies that for the above choice of $\epsilon$,
\[
  \E_{b \sim \Bc}\sbr*{\E_{\chi_b} f_b(\Dc_k(h(f_b + \chi_{b})))} \geq \epsilon^2/64  = \frac{1}{2^{14} k^{2(\alpha - 1)/\alpha}}\,.
\]
This finishes the proof of the theorem. Note that the readability of the proof was prioritized over optimality and it is possible to obtain significantly better constants. \qed

\section{Non-convex Lower Bound (Proof of Theorem~\ref{thm:lower-bound-nonconvex})}\label{sec:non-convex-lower-proof}

The proof is based on the proof of Theorem 1 in \cite{arjevani2019lower}. The only difference is that we assume bounded $\alpha-$moment of the stochastic oracle instead of bounded variance as in the original proof. We refer readers to \cite{arjevani2019lower} for more backgrounds and intuitions. For convenience, we study the stochastic setting ($K=1$ in \cite{arjevani2019lower}) instead of batched setting. We denote a $d-$dimensional vector $x$ as, $x = [x_{(1)}; ...; x_{(d)}].$
Let $\text{support}(x)$  denote the set of coordinates where $x$ is nonzero, i.e.
\begin{align*}
\text{support}(x) = \{i \in [d] | x_{(i)} \neq 0\} \subseteq [d].
\end{align*}
Denote $\text{prog}_\beta(x)$ as the highest index whose entry is $\beta-$far from zero.  
\begin{align*}
\text{prog}_\beta(x) = \max \{i \in [d] | |x_{(i)}| > \beta \} \in [d].
\end{align*}
Note that the function $\text{prog}_\beta(\,\cdot\,)$ is decreasing in $\beta$.
The function we use to prove the theorem is the same as in \cite{arjevani2019lower,carmon2017lower}. We denote
\begin{align*}
f_d(x) = -\Psi(1)\Phi(x_{(1)}) + \sum_{i = 2}^d \rbr*{\Psi(-x_{(i-1)})\Phi(-x_{(i)}) - \Psi(x_{(i-1)})\Phi(x_{(i)})}, \ \text{ where}
\end{align*}
\begin{align*}
\Psi(x) = \begin{cases}0, \quad x\le 1/2 \\ \exp(1 - \frac{1}{(2x-1)^2}), x > 1/2
\end{cases}, \Phi(x) = \sqrt{e}\int^x_{-\infty} e^{-\frac{t^2}{2}}dt.
\end{align*}
The above function satisfies the following important properties,
\begin{lemma}[Lemma 2 in \cite{arjevani2019lower}]\label{lemma:carmon}
	The function $f_d$ satisfies the following properties,
	\begin{enumerate}
		\item $f_d(0) - \inf_x f_d(x) \le 12d$.
		\item $f_d$ is $L_0$-smooth, where $L_0 = 152$.
		\item For all $x$, $\|\nabla f_d(x)\|_\infty \le 23$.
		\item For all $x$, $\text{prog}_0(\nabla f_d(x)) \le \text{prog}_\frac{1}{2}(x) + 1$
		\item For all $x$, if $\text{prog}_1(x) < d$, then $\|\nabla f_d(x)\|^2 \ge 1$.
	\end{enumerate}
\end{lemma}
We also define the stochastic oracle $g_d(x)$ as below
\begin{align*}
g_d(x)_{(i)} = \rbr*{1 + \mathbbm{1}{\cbr*{i = \text{prog}_{\frac{1}{4}}(x) + 1}}\rbr*{\frac{z}{p} - 1} }\frac{\partial}{\partial x_{(i)}} f_d(x) 
\end{align*}
where $z \sim \text{Bernoulli}(p)$. The stochasticity of $g_d(x)$ is only in the $(\text{prog}_{\frac{1}{4}}(x) + 1)$th coordinate. It is easy to see that $g_d(x)$ is a probability-$p$ zero chain as in \citep[Definition 2]{arjevani2019lower} i.e. it satisfies
\begin{align*}
\mathbb{P}\rbr*{\exists x, \text{ s.t. prog}_0(g_d(x)) = \text{prog}_{\frac{1}{4}}(x) + 1 } \le p,\\
\mathbb{P}\rbr*{\exists x, \text{ s.t. prog}_0(g_d(x)) > \text{prog}_{\frac{1}{4}}(x) + 1 } = 0.
\end{align*}
The second claim is because $\text{prog}_\beta(\,\cdot\,)$ is decreasing in $\beta$ and
\[
\text{prog}_{\frac{1}{4}}(\nabla f_d(x)) \leq \text{prog}_{0}(\nabla f_d(x)) \leq \text{prog}_{\frac{1}{2}}(x) + 1 \leq \text{prog}_{\frac{1}{4}}(x) + 1\,.
\]
The first claim is because if $z = 0$, then we explicitly set the $(\text{prog}_{\frac{1}{4}}(x) + 1)$th coordinate to 0. The stochastic gradient additionally has bounded $\alpha$-moment as we next show.
\begin{lemma}\label{lem:lowerbound-variance}
	The stochastic oracle above is an unbiased estimator of the true gradient, and for any $\alpha \in (1,2]$
	$$ \E\sbr*{\|g_d(x)\|^\alpha}  \le 2\|\nabla f_d(x)\|^\alpha + 23^\alpha \frac{2}{p^{\alpha-1}} .$$
\end{lemma}
\begin{proof}
	The unbiased-ness is easy to verify. For the bounded $\alpha$-moment, observe that only the $(\text{prog}_{\frac{1}{4}} + 1)$-th coordinate is noisy and differs by a factor of $(\frac{z}{p} - 1)$. Hence, we have
	\begin{align*}
	\E\sbr*{\|g_d(x)\|^\alpha}  &\le  2\|\nabla f_d(x)\|^\alpha + 2\E\sbr*{\|g_d(x) - \nabla f_d(x)\|^\alpha} \\
	&\le 2\|\nabla f_d(x)\|^\alpha + \|\nabla f_d(x)\|_\infty^\alpha \E\sbr*{|\frac{z}{p} - 1|^\alpha}  \\
	&\le 2\|\nabla f_d(x)\|^\alpha + \|\nabla f_d(x)\|_\infty^\alpha\frac{p(1-p)^\alpha + (1-p)p^\alpha}{p^\alpha} \\
	&\le 2\|\nabla f_d(x)\|^\alpha + 23^\alpha \frac{2}{p^{\alpha-1}} 
	\end{align*}
	The first inequality followed from Jensen's inequality and the convexity of $\norm{\, \cdot \,}^\alpha$ for $\alpha \in (1,2]$:
	\[
	\norm{u + v}^\alpha \leq 4\norm{\tfrac{u + v}{2}}^\alpha  \leq 2(\norm{u}^\alpha + \norm{v}^\alpha) \text{ for any } u,v\,.
	\]
\end{proof}

Now we are ready to prove Theorem~\ref{thm:lower-bound-nonconvex}. Given accuracy parameter $\epsilon$, suboptimality $\Delta = f(0) - f^*$, smoothness constant $L$, and bounded $\alpha-$moment $G^\alpha$, we define
$$f(x) = \frac{L\lambda^2}{152}f_d(\frac{x}{\lambda}),$$
where $\lambda = \frac{304\epsilon}{L}$ and $d = \lfloor\frac{\Delta L}{7296\epsilon^2}\rfloor$. Then,
\[
g(x) = \frac{L\lambda}{152}g_d(x/\lambda) = 2\epsilon g_d(x/\lambda).
\]
Using Lemma~\ref{lem:lowerbound-variance}, we have 
\[
\E\sbr*{\|g(x)\|^\alpha} \le  8\epsilon^\alpha\|\nabla f_d(x)\|^\alpha + \frac{5000\epsilon^\alpha}{p^{\alpha-1}} 
\]
When $G \ge 4\sqrt{\Delta L}, $ we can set $p = \frac{(5000\epsilon)^{\frac{\alpha}{\alpha-1}}}{(G-4\sqrt{\Delta L})^{\frac{\alpha}{\alpha-1}}}$ and get $\E\sbr*{\|g(x)\|^\alpha} \le G^\alpha.$

Let $x_k$ be the output of any \emph{zero-respecting} algorithm $\Ac$. By \citep[Lemma 1]{arjevani2019lower}, we know that with probability at least 1/2, $ \text{prog}_1(x_{k}) \leq \text{prog}_0(x_{k}) < d$ for all $k \le \frac{(d-1)}{2p}$. Now applying Lemma~\ref{lemma:carmon}.5, we have that for all $k \leq \frac{(d-1)}{2p}$:
\[
\E[\norm{\nabla f(x_k)}] \geq \frac{1}{2}\frac{L\lambda}{152}\E[\mathbbm{1}{\|\nabla f_d(x_k/\lambda)} \|\,|\, \cbr*{\text{prog}_{1}(x_k) < d}] \ge \epsilon\,.
\]
Therefore, $\E\|\nabla f(x_k)\| \ge \epsilon$, for all $k \le \frac{(d-1)}{2p} = \frac{(G-4\sqrt{\Delta L})^{\frac{\alpha}{\alpha-1}}\Delta L}{7296 \times 5000^{\frac{\alpha}{\alpha-1}}\epsilon^{2 + \frac{\alpha}{\alpha-1}}} = c(\alpha) (G-4\sqrt{\Delta L})^{\frac{\alpha}{\alpha-1}}\Delta L \epsilon^{-\frac{3\alpha-2}{\alpha-1}}$. By eliminating $\epsilon$, we can rewrite this in terms of $k$. Finally, the techniques from \citep[Theorem 3]{arjevani2019lower} show how to lift lower-bounds for \emph{zero-respecting} algorithms to any randomized method.

\section{A Comparison with \citep{simsekli2019tail}}\label{sec:comp}

We are not the first to study the heavy-tailed noise behavior in neural network training. The novel work by \citet{simsekli2019tail} studies the noise behavior of AlexNet on Cifar 10 and observed that the noise does not seem to come from Gaussian distribution. However, in our AlexNet training with ImageNet data, we observe that the noise histogram looks Gaussian as in Figure~\ref{fig:ICML-comp}(a, b). We believe the difference results from that in \citep{simsekli2019tail}, the authors treat the noise in each coordinate as an independent scaler noise, as described in the original work on applying tail index estimator. We on the other hand, consider each the noise as a high dimensional random vector computed from a minibatch. We are also able to observe heavy tailed noise if we fix a single minibatch and plot the noise in each dimension, as shown in Figure~\ref{fig:ICML-comp}(c). The fact that noise is well concentrated on Cifar is also observed by \citet{panigrahi2019non}.

\begin{figure}[h]
	\begin{subfigure}{.33\textwidth}
		\centering
		\includegraphics[width=0.95\linewidth]{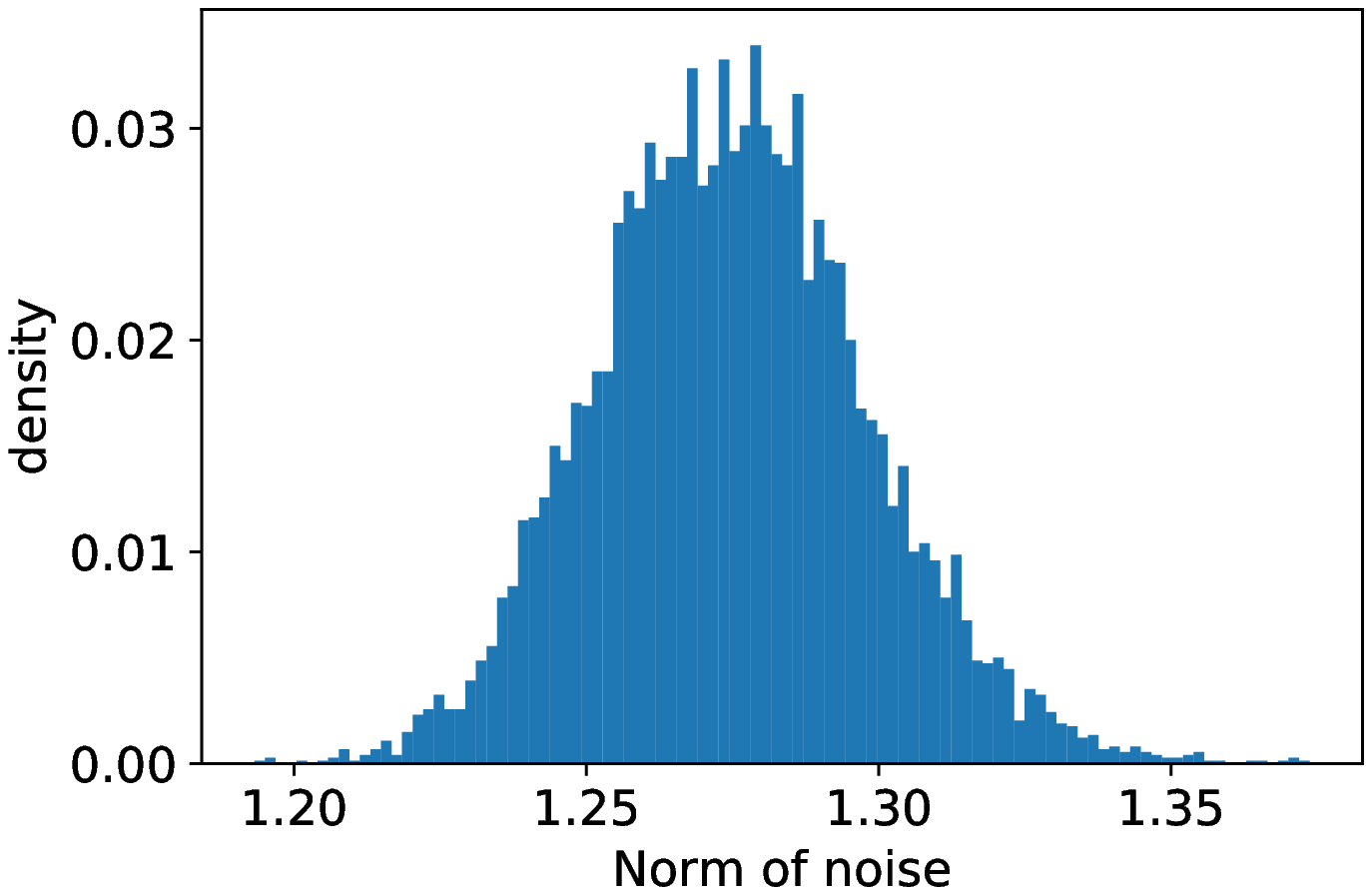}
		\caption{}
	\end{subfigure}
	\begin{subfigure}{.33\textwidth}
		\centering
		\includegraphics[width=0.95\linewidth]{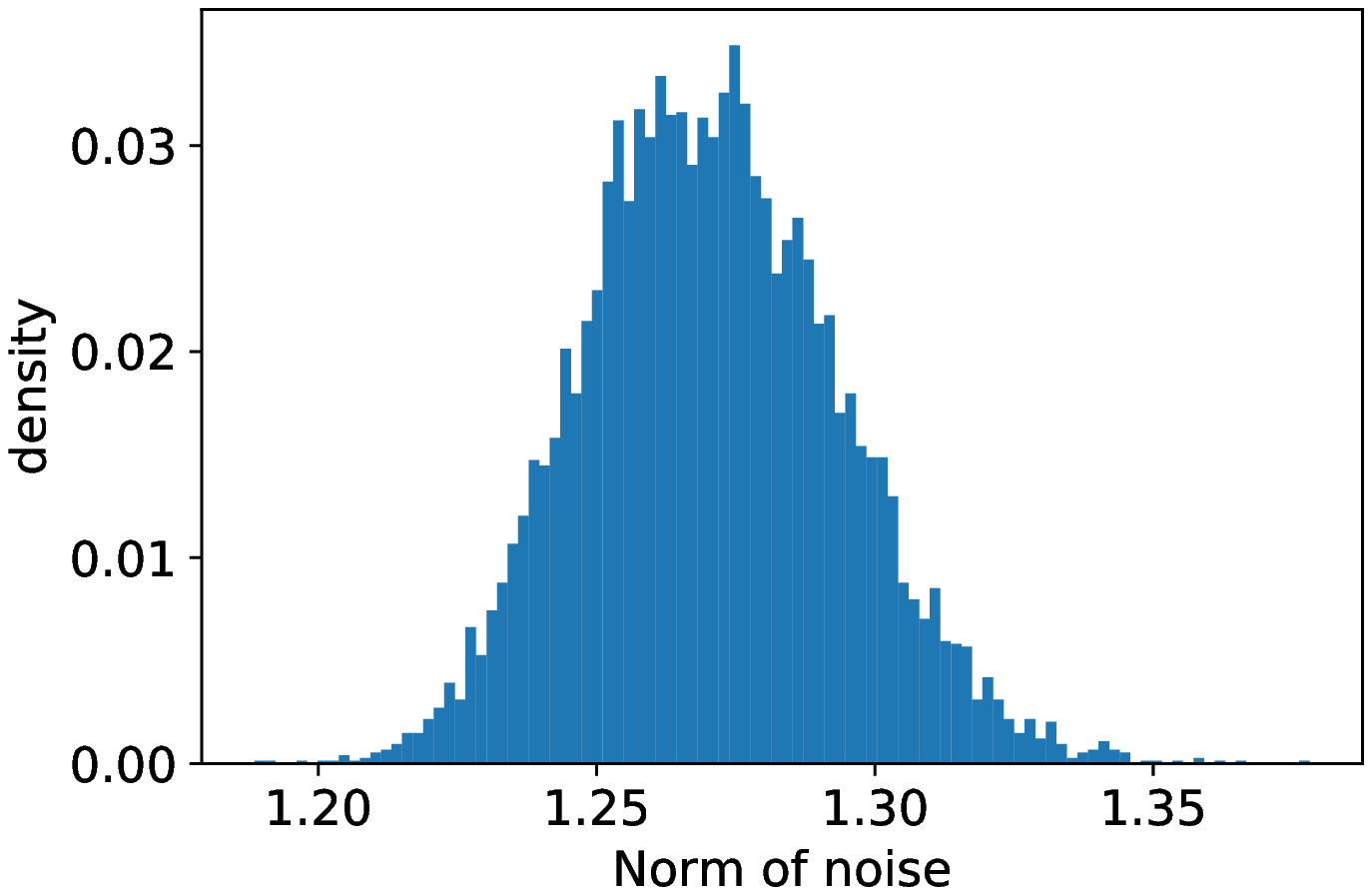}
		\caption{}
	\end{subfigure}
	\begin{subfigure}{.33\textwidth}
		\centering
		\includegraphics[width=0.95\linewidth]{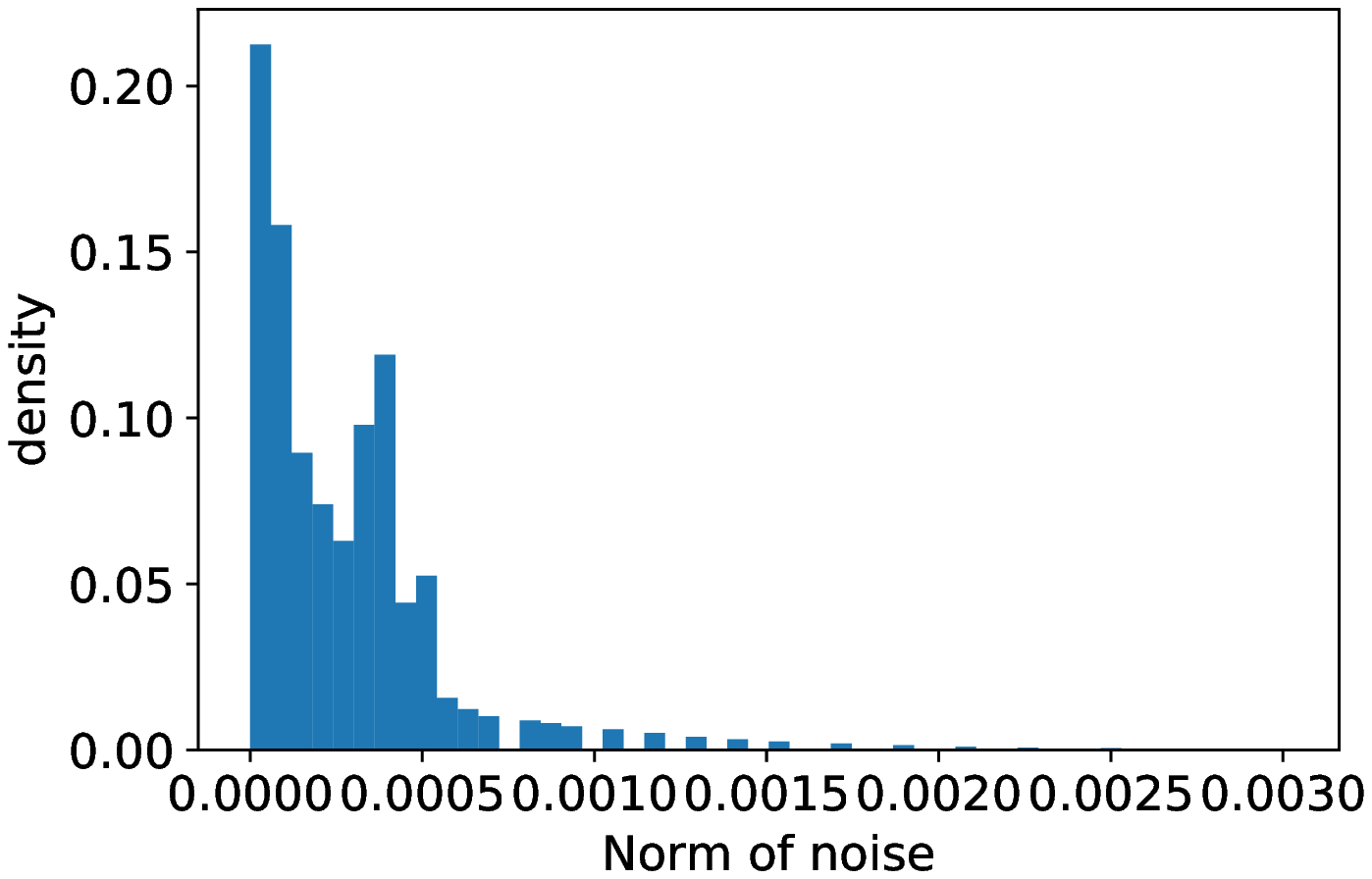}
		\caption{}
	\end{subfigure}
	
	\caption{ (a) Noise histogram of AlexNet on ImageNet data at initialization. (b)Noise histogram of AlexNet on ImageNet data at 5k iterations.   (c) The per dimension noise distribution within a single minibatch at initialization. }
	\label{fig:ICML-comp}
\end{figure}

Furthermore, we used the tail index estimator presented in \cite{simsekli2019tail} to estimate the tail index of noise norm distribution. Though some assumptions of the estimator are not satisfied (in our case, the symmetry assumption; in \cite{simsekli2019tail}, the symmetry assumption and independence assumption), we think it can be an indicator for measuring the ``heaviness'' of the tail distribution.

\begin{figure}[t]
\centering
\captionsetup[subfigure]{position=b,format=myformat}
    \begin{subfigure}{.45\textwidth}
	    \centering
    	\includegraphics[width=0.85\linewidth]{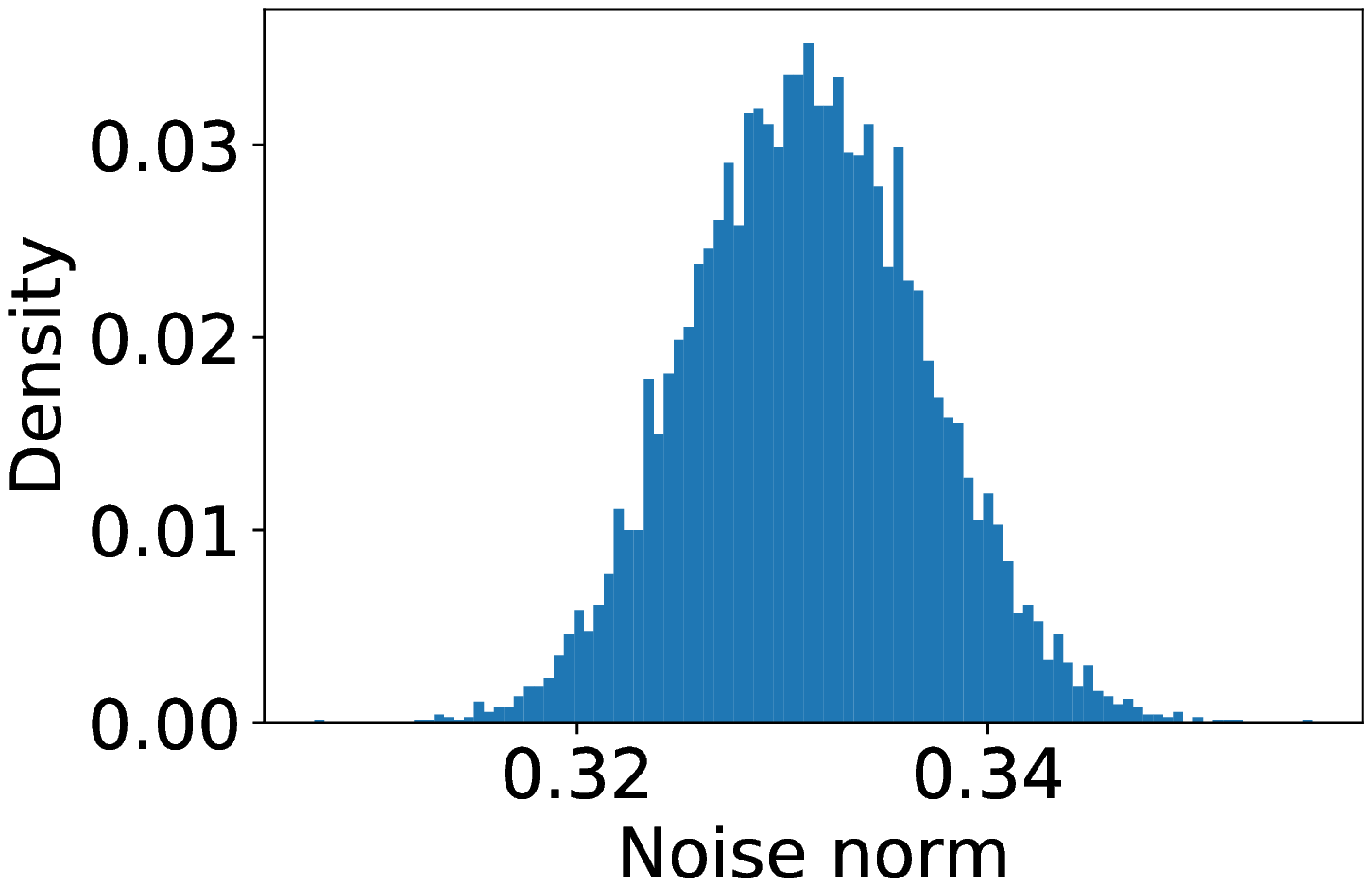}
        \caption{$\text{ImageNet training}, \hat{\alpha}= 1.99$}
    \end{subfigure}
    \begin{subfigure}{.45\textwidth}
	    \centering
    	\includegraphics[width=0.85\linewidth]{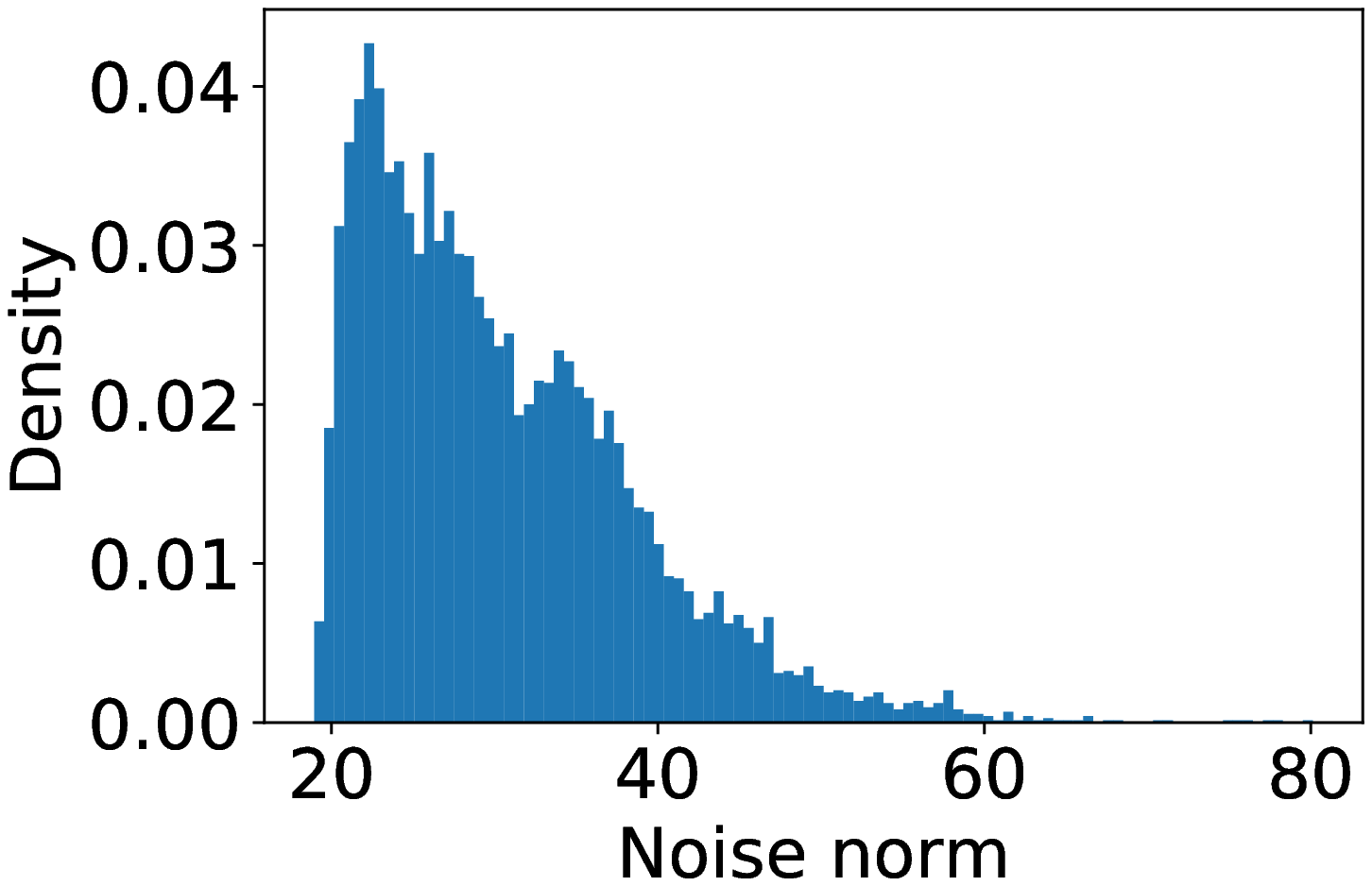}
	    \caption{Bert pretraining, $\hat{\alpha}=1.08$}
    \end{subfigure}
	\vspace{-0.2cm}
	\caption{Tail index estimation of gradient noise in ImageNet training and BERT training.}
\end{figure}

\section{ACClip in ImageNet Training}\label{sec:renet-exp}

For completeness, we test ACClip on ImageNet training with ResNet50. After hyperparameter tuning for all algorithms, ACClip is able to achieve better performance compared to ADAM, but worse performance compared to SGD. This is as expected because the noise distribution in ImageNet + ResNet50 training is well concentrated. The validation accuracy for SGD, ADAM, ACClip are $0.754, 0.716, 0.730$ respectively.

\begin{figure}[h]
\centering

\begin{subfigure}[b]{0.55\textwidth}
	\centering
	\includegraphics[width=\linewidth]{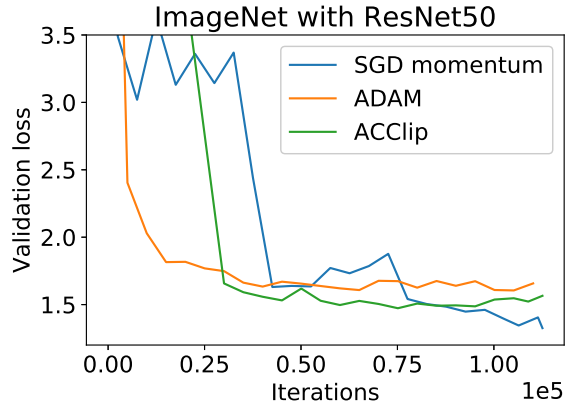}
	\caption{}\label{fig:loss-sgd-Adam-Acclip-imagenet}
\end{subfigure}
	\caption{ Validation loss for ResNet50 trained on ImageNet. SGD outperforms Adam and ACClip.}
\end{figure}